\documentclass[12pt, a4paper,reqno]{amsart}
\usepackage[a4paper, margin=3cm]{geometry}
\usepackage{amsmath}
\usepackage{amssymb,amsfonts}
\usepackage{graphicx}
\usepackage{tikz-cd}
\usepackage{tikz}

\usepackage{color}
\usepackage[colorlinks=true,linkcolor=blue,citecolor=blue,pdfpagelabels=false]{hyperref}
\usepackage{cleveref}
\usepackage{url}

\newtheorem{Theorem}{Theorem}[section]

\theoremstyle{definition}

\newtheorem{thm}{Theorem}[section]
\newtheorem{cor}[thm]{Corollary}
\newtheorem{lem}[thm]{Lemma}
\newtheorem{prop}[thm]{Proposition}
\newtheorem{defn}[thm]{Definition}

\newtheorem{rem}[thm]{Remark}
\numberwithin{equation}{section}

\def\C{\mathbb C}

\def\N{\mathbb N}

\def\R{\mathbb R}

\def\Z{\mathbb Z}
\def\K{\mathbb K}
\def\G{\mathbb G}
\def\P{\mathbb P}

\newcommand{\CB}{\mathcal{B}}

\newcommand{\CH}{\mathcal{H}}

\newcommand{\CL}{\mathcal{L}}

\newcommand{\CP}{\mathcal{P}}

\def\be{\begin{equation}}
	\def\ee{\end{equation}}
\def\bt{\begin{Theorem}}
	\def\et{\end{Theorem}}
\def\bi{\begin{itemize}}
	\def\ei{\end{itemize}}
\def\bea{\begin{eqnarray}}
	\def\eea{\end{eqnarray}}
\def\beast{\begin{eqnarray*}}
	\def\eeast{\end{eqnarray*}}
\def\ben{\begin{enumerate}}
	\def\een{\end{enumerate}}

\def\bi{\bibitem}

\newcommand{\norm}[1]{\left\Vert#1\right\Vert}

\newcommand\inner[2]{\left\langle #1, #2 \right\rangle}

\def\rar{\rightarrow}

\renewcommand{\MR}[1]{} 

\newcommand{\abs}[1]{\left\vert#1\right\vert}

\allowdisplaybreaks[4]

\begin{document}
	\title[Neveu decomposition and stochastic ergodic theorems]{ On the non-commutative Neveu decomposition and stochastic ergodic theorems }

	\author[Bikram]{Panchugopal Bikram}
	\author[Saha]{Diptesh Saha}

	\date{\today}
	
	\address{School of Mathematical Sciences,
		National Institute of Science Education and Research,  Bhubaneswar, An OCC of Homi Bhabha National Institute,  Jatni- 752050, India}
	
	\email{bikram@niser.ac.in} \email{diptesh.saha@niser.ac.in}

	
	\begin{abstract} 
	In this article,  we prove  Neveu decomposition for the action of locally compact amenable semigroup of positive contractions on semifinite von Neumann algebras and thus, it entirely resolves the problem for the actions of arbitrary amenable semigroup on semifinite von Neumann algebras. We also prove it for amenable group actions by Markov automorphisms on  any $\sigma$-finite von Neumann algebras.
	As an application, we obtain stochastic ergodic theorem for actions of $ \Z_+^d$ and $\R_+^d$ for $ d \in \N$ by positive contractions on  $L^1$-spaces associated with a finite von Neumann algebra. It yields the first ergodic theorem for positive contraction on non-commutative $L^1$-spaces  beyond the Danford-Schwartz category. 
\end{abstract}

	\maketitle 
	
	\section{Introduction}

	The connection between von Neumann algebra and ergodic theory is well known in the literature. This article falls in the conjunction of these two well studied areas of research. Especially, we study the Neveu decomposition and stochastic ergodic theorems for actions by positive contractions on non-commutative $L^1$-spaces. 
	
	\vspace{0.3cm}	
	Although the study of ergodic theorems originated in the classical mechanics but has wide  applications in modern day mathematics and physics. Historically, the subject begins with  the mean ergodic theorem by von Neumann \cite{neumann1932proof} (which is the convergence  of ergodic averages associated to a contraction in the Hilbert space norm) and the pointwise ergodic theorems by Birkhoff \cite{birkhoff1931proof} (which is the almost everywhere convergence  of ergodic averages associated to measure preserving transformations in $L^p$ spaces). Later it becomes an 
extensively studied field  of research and still remains active. 
	
	\vspace{0.2cm}	
	The main focus of this article is to prove non-commutative Neveu decomposition and as an application to obtain stochastic ergodic theorem (that is convergence in measure). To motivate, we begin with a brief history of pointwise ergodic theorems. After its inception in 1939, pointwise ergodic theorems has seen a lot of generalisations in both classical and non-commutative settings.

\vspace{0.2cm}	
Classically, given a measure preserving system $(X,T, \mu)$, Birkhoff's ergodic theorem states that the ergodic averages associated to the Koopman operator converges almost everywhere to the conditional expectation onto the fixed point space. It is a natural question to ask Birkhoff's ergodic theorem beyond Koopman operator, such as  whether this result holds true for general positive contractions on $L^1$-spaces.
	 
\vspace{0.2cm}
A partial answer is obtained by Hopf, Dunford and Schwartz. An operator $S: L^1 + L^\infty \to L^1 + L^\infty$ is called Dunford-Schwartz operator if it is 
	  a $L^1$-$L^\infty$ contraction.  
	  In \cite{hopf1954general} and \cite{dunford1988linear}, the authors considered any  Dunford-Schwartz operator $S$ associated to a general measure space $(X, \mu)$ and proved the ergodic averages $\frac{1}{n} \sum_{0}^{n-1} S^k(f)$ converge almost everywhere for all $f \in L^1(X, \mu)$.

	\vspace{0.2cm}	
	Although this trend seems promising, but a further extension of the pointwise ergodic theorems for more general positive contraction on $L^1(X, \mu)$ need   not be true. For example, in \cite{chacon1964class}, the authors constructed a class of isometry such that for each such isometry $T$, there exists an $f \in L^1(X, \mu)$ such that the limit of the sequence $\frac{1}{n} \sum_{0}^{n-1} T^k(f)$ fails to exist almost everywhere. Furthermore, in \cite{ionescu1965category}, Tulcea  showed that there are enough examples of positive isometric isomorphisms on $L^1[0, 1]$ for which pointwise ergodic theorem does not hold. But on the contrary, 
	pointwise ergodic theorem holds for positive contractions on  $L^p$-spaces for $ 1< p < \infty $. Indeed, in \cite{akcoglu1975pointwise}, Akcoglu proved the celebrated pointwise convergence result for a positive contraction on classical $L^p$-spaces for $1<p<\infty$.
	Thus, it is natural to look for a satisfactory convergence result for positive contractions on $L^1$-spaces. 
\vspace{0.2cm}	

In 1985, Krengel proved the following:  given a $\sigma$-finite measure spaces $(X, \mu)$ and a positive contraction $T$ defined on $L^1(X, \mu)$,  the sequence $\frac{1}{n} \sum_{k=0}^{n-1} T^k(f)$ converges in measure for all $f \in L^1(X, \mu)$. Actually, Krengel proved it for $d$-many commuting positive contractions on $L^1(X, \mu)$ \cite[ see Theorem 3.4.9]{Krengel1985}. 
\vspace{0.2cm}	
	
 Indeed, Krengel used an elegant  machinery to prove it. In fact, the  main technique of Krengel's theorem is Neveu decomposition. Given a positive contraction $T$ on the space $L^1(X, \mu)$, Neveu decomposition essentially breaks down the space  $X$ into two disjoint sets, determined uniquely upto measure zero sets such that one of them being the support of a positive $T$-invariant function $f \in L^1(X, \mu)$ and the other one being the support of a weakly wandering function $h \in L^\infty(X, \mu)_+$. For more details and relevant definitions we refer to \cite[Theorem 3.4.6]{Krengel1985}.
\vspace{0.2cm}

 Neveu decomposition is also an independent subject of research. Typically, given a $\sigma$-finite measure space $(X, \mu)$ and a transformation $T$ defined on it, the problem of finding a $T$-invariant  finite measure is studied extensively in the literature and simultaneously many necessary and sufficient conditions are obtained in the process. In \cite{Hajian1964}, the authors characterised the existence of finite invariant measure with the non existence of weakly wandering set of strictly positive measures. The result is then extended to arbitrary group of non-singular transformations by Hajian and Ito in \cite{Hajian1968/1969}. On the other hand, given a positive contraction $T$ on $L^1(X, \mu)$, a similar question can be asked. Moreover, observe that the existence of a finite measure $\nu$ which is absolutely continuous with respect to $\mu$, is equivalent to the existence of a  positive $T$-invariant $f$ in $L^1(X, \mu)$. The condition for existence of finite invariant measure in this situation is first studied by Ito \cite{ito1964invariant}. Later on, Neveu and Krengel obtained a proper decomposition of measure space that we have mentioned in the previous paragraph.
\vspace{0.2cm}	

In the non-commutative setting,  a measure space is usually replaced by a von Neumann algebra. Although many necessary and sufficient conditions regarding the existence of finite measure is present in the literature but in the non-commutative setup very little was known until Grabarnik and Katz \cite{MR1336332}. They established Neveu decomposition for finitely many commuting tuples of $*$-automorphisms acting on a finite von Neumann algebra. Recently, in \cite{Bik-Dip-neveu} the authors obtained Neveu decomposition for the actions of  amenable group by  $*$-automorphisms on a finite von Neumann algebra.
\vspace{0.2cm} 

It seems the  known techniques has limitation to generalize the Neveu decomposition for positive contractions and for semifinite von Neumann algebras. 
In this article, we prove Neveu decomposition for the   actions of any amenable semigroup of positive contractions on semifinite  von Neumann algebras. Further, we will also prove  the Neveu decomposition  for  the actions of amenable group which is  compatible with the modular automorphism group associated with a weight on  a $\sigma$-finite von Neumann algebra.  
\vspace{0.2cm}

Let $(M, G, \alpha)$ be a non-commutative dynamical system, where $M$ is a von Neumann algebra  with a faithful, normal tracial state $\tau$ and $\alpha$ is an action of an amenable group $G$ on $M$ by $*$-automorphisms. In \cite{Bik-Dip-neveu}, we showed that the existence of a maximal invariant state can be characterised by an auxiliary infimum condition. In particular, if $\rho$ is the maximal invariant state in this case, then it was shown that the support of $\rho$ is the maximal projection such that for any non-zero subprojection $q$ of support of $\rho$, $\inf_{n \in \N} \tau(A_n(q))>0 $. On the other hand, it was also proved that if for some non-zero projection $p \in M$, $\inf_{n \in \N} \tau(A_n(p))=0$, then there exists a non-zero subprojection $q$ of $p$ such that $\lim_{n \to \infty} \norm{A_n(q)}=0$, that is $q$ is a weakly wandering projection. In the later part, we in effect implicitly found a non-zero projection $q$ and a sequence $\{g_1, \cdots, g_n, \cdots\} \subseteq G$ such that $\alpha_{g_i}(q) \perp \alpha_{g_j}(q)$ whenever $i \neq j$. Note that since $\alpha_{g_i}$'s are $*$-automorphisms, $\alpha_{g_i}(q)$'s are also projections. Furthermore, the existence of tracial state is also heavily used to find the projection and the sequence mentioned above. For more details and rigorous proofs of this facts we refer to \cite[Section 3]{Bik-Dip-neveu}. To prove a similar result, when $(M, G, \alpha)$ is a non-commutative dynamical system ( See Definition \ref{action on Banach sp} ), one has to first overcome the aforementioned technical difficulties. In this article, we essentially take a different approach. With the help of Lemma \ref{main lem-semifinite2} and Lemma \ref{req cpt set}, we inherently proved that there exits a non-zero projection $ q \in M$ such that 
\begin{align*}
	\lim_{n \rar \infty }\mu ( A_n(q)) = 0, \text{ for all } \mu \in M^*.
\end{align*}
Then a version of Mean ergodic theorem for Banach spaces (Theorem \ref{mean erg thm}) is implemented to show that $q$ is weakly wandering.
\vspace{0.2cm}

The second half of the article is dedicated to obtain the  Krengel's stochastic ergodic theorem for actions by some semigroup of positive contractions on non-commutative $L^1$-spaces. The main tool of this proof is the Neveu decomposition and together with a version of pointwise ergodic theorem on the corner  where there is an  invariant state according to the Neveu decomposition.  Equivalently, we need to prove a non-commutative pointwise ergodic theorem for a non-commutative dynamical system associated with an amenable semigroup preserving a faithful normal state on a von Neumann algebra. 
\vspace{0.2cm}

We have already discussed about the pointwise ergodic theorem for classical $L^1$-spaces for a single positive contraction. Further generalisation of this result for various group actions and other $L^p$-spaces ($1< p < \infty$) are also hugely studied in the literature. For action of amenable groups, Lindenstrauss \cite{Lindenstrauss2001} proved that the ergodic averages associated to tempered F\o lner sequence converges almost everywhere for all $f$ in classical $L^1$-spaces. For more information and other generalizations and state-of-the-art results we refer to \cite{krause2022pointwise}, \cite{anantharaman2010theoremes}, \cite{Lindenstrauss2001},  \cite{Calderon-ergodic}, \cite{tao-ergodic} and the references therein. 
\vspace{0.2cm}

For the non-commutative setup these results are also extensively studied. The study of non-commutative  ergodic theorems were initiated by the pioneering work of Lance \cite{Lance1976}. In this article, a pointwise ergodic theorem is studied on a von Neumann algebra. This result is then improved substantially in the works of Yeadon, Kummerer, Conze, Dang-Ngoc and many others [see, \cite{Yeadon1977}, \cite{Yeadon1980}, \cite{Kuemmerer1978}, \cite{Conze1978} and the references therein]. Particularly, Yeadon in  \cite{Yeadon1977} and \cite{Yeadon1980} extended the results of Lance \cite{Lance1976} to the non-commutative $L^1$-spaces. After that, Junge and Xu \cite{Junge2007} extended Yeadon's result to prove a pointwise ergodic theorem in non-commutative $L^p$-spaces (both tracial and non-tracial) for $1< p < \infty$ for some specific $\Z^d_+$ of $\R^d_+$ ($d \geq 1$) actions. In recent times, these results are further generalised for the actions of locally compact  groups with polynomial growth (in this case as earlier the actions are assumed to be sub-tracial) in \cite{Hong2021} on tracial non-commutative $L^p$-spaces ($1\leq p < \infty$). Furthermore, using the Calderon-Zygmund decomposition and variaous norm estimates of the elements of non-commutative $L^1$-spaces, this result is further extended for actions of amenable groups and the averages associated to some filtered F\o lner sequence in \cite{cadilhac2022noncommutative}. Recently, in  \cite{samya-lamperti2023}, the authors established  first   ergodic theorem for  large class of contractions beyond the Danford-Schwartz category on  non-commutative $ L^p$-spaces for $ 1 < p < \infty$.

\vspace{0.2cm}

Thus, available results regarding  pointwise ergodic theorems on non-commutative $L^1$-spaces  fall in the Dunford-Schwartz category.  So far there is no ergodic theorem for positive contraction on non-commutative $L^1$-spaces. Indeed, in the literature it is known to be challenging to obtain an ergodic theorem for semigroup of positive contractions on non-commutative $L^1$-spaces, even for a single positive contraction and it is anticipated that it will require  non-trivial new approaches other than \cite{Junge2007}, \cite{Hong2021}, \cite{Bik-Dip-neveu} etc.
\vspace{0.2cm}

In this article, we obtain a satisfactory ergodic theorem for positive $L^1$-contractions associated with a tracial state. In fact, to prove it, the first non-trivial difficulty  was to prove Neveu decomposition for positive contraction and the second one was to prove a pointwise ergodic theorem for state preserving positive  contractions. 
In our proof, we extensively use the Neveu decomposition and a version of pointwise ergodic theorem which is mainly proved in our previous article \cite{Bik-Dip-neveu} and \cite{bikram2023noncommutative}.
\vspace{0.2cm}

Now we  highlight some of our main results to  get more idea about it.  
Let $M$ be a $\sigma$-finite von Neumann algebra and $G$ be a locally compact, second countable, Hausdorff semigroup with a both left and right invariant $\sigma$-finite measure $m$ and left-right F\o lner net $
\{K_l\}_{l \in \R_+}$. Furthermore, let $\alpha$ be an action of $G$ on $M$ by positive contractions, then we have the following Neveu decomposition.

\begin{thm}[Neveu Decomposition]\label{into-Neveu}
	Let $ (M, G, \alpha)$ be as above and suppose it falls one of the following category.
	\begin{itemize}
		\item[(a)] $M$ be a semifinite von Neumann algebra and $( M, G , \alpha) $ be a non-commutative dynamical system, i.e, $G$ is a  semigroup and $\alpha$  is is the action by positive contractions, or 
		\item[(b)] $( M, G , \alpha) $ be a Markov covariant system, i.e, $G$ is a group and the action commute with the modular automorphisms group associated to a f.n weight on $M$. 
	\end{itemize}
	Then there  exist two projections $e_1, e_2 \in M$ such that $e_1 + e_2 = 1$ and 
	\begin{enumerate}
		\item there exists a $G$-invariant normal state $\rho$ on $M$  with support $s(\rho) = e_1$ and
		\item there exists a weakly wandering operator $x_0 \in M$ with support $s(x_0)= e_2$, i.e,
		$ \frac{1}{m(K_n) } \int_{K_n} \alpha_g(x_0) d m(g)$ converges to $0$ in norm as $ n \rar \infty $. 
	\end{enumerate}
	Further, in both the cases $s(\rho)$ and $s(x_0)$ are  unique. 
\end{thm}	
\vspace{0.2cm}

We note that the proof of the Neveu decomposition for  	$( M, G , \alpha) $, when it falls in the second category use the first category by lifting it to the cross product of $M$ with it's modular automorphisms group, as the cross product of $M$ with it's modular automorphisms group is a semifinite von Neumann algebra. 
\vspace{0.2cm} 	

Now let $M$ be a finite von Neumann algebra and  note that $( M, \Z^d_+, \alpha )$ is determined by  $d$-commuting positive contractions  $ \alpha_1, \alpha_2, \cdots ,\alpha_d$ on $M$ such that $$ \alpha_{( i_1, \cdots, i_d)} (\cdot) = \alpha_1^{i_1} \alpha_2^{i_2}\cdots \alpha_d^{i_d} (\cdot) \text{ for }  ( i_1, \cdots, i_d) \in \Z_+^d.$$ For the dynamical system 
$(M, \R^d_+, \alpha)$, we consider the ergodic avarages with respect to the set  $Q_a:= \{(t_1, \ldots, t_d) \in \R^d_+ : t_1<a, \ldots, t_d< a\}$  for $a \in \R_+$. Thus, with the preceding notations, consider following ergodic averages;

\begin{align*}
	A_a(\cdot):= 	
	\begin{cases}
		\frac{1}{a^d} \sum_{0 \leq i_1 <a} \cdots  \sum_{0 \leq i_d <a} \alpha_1^{i_1} \cdots \alpha_d^{i_d} (\cdot)&\text{ when  } G = \Z_+^d,~ a \in \N, \\\\
		\frac{1}{a^d} \int_{Q_a} \alpha_t(\cdot) dt &  \text{  when } G = \R_+^d, ~a \in \R_+.
	\end{cases}
\end{align*}	

Then following stochastic ergodic theorem

\begin{thm}[Stochastic Ergodic Theorem]\label{stocha-into}  With the preceding notations, let $e_1, e_2 \in M$ be projections obtained in Neveu decomposition  \ref{into-Neveu}. Then we have the following results.  
	\begin{enumerate}
		\item[(i)] For all $B \in L^1(M_{e_1}, \tau_{e_1})$, there exists $\bar{B} \in L^1(M_{e_1}, \tau_{e_1})$ such that $A_n^*(B)$ converges b.a.u to $\bar{B}$. Moreover,  $A_n(B)$ converges in measure to $\bar{B}$.
		\item[(ii)] For all $B \in L^1(M_{e_2}, \tau_{e_2})$, $A_n^*(B)$ converges to $0$ in measure.
	\end{enumerate}
\end{thm}

Here we refer to \S 2 for the definition of b.a.u and stochastic convergence and basic facts of non-commutative $L^p$-spaces.
Further, we point out that if for all $g \in G$, $ \alpha_g^* : L^1(M, \tau) \rar L^1(M, \tau) $ is Lamperti positive contraction (i.e,  $\alpha_g^*(A) \alpha_g^*(B ) = 0$, whenever $A, B \in L^1(M, \tau)_+ $ with $AB = 0$), then  we have a stronger version of  stochastic ergodic theorem. Indeed, we have the following .

	\begin{thm}[Strong Stochastic Ergodic Theorem]\label{conv in e1-e2 corner} Let $M$, $G$ and $\alpha$ as in the Theorem \ref{stocha-into} and assume that for  each $g \in G$, $\alpha_g^*$ is lamperti operator.  	Let $X \in L^1(M, \tau)$, then there exists $\overline{X} \in L^1(M, \tau)$ such that $A_n(X)$ converges to $\overline{X} $ in measure. Further, suppose 
	$e_1, e_2 \in M$ are as in the Theorem \ref{into-Neveu},  then  $e_1   \overline{X}e_1 = \overline{X} $ and $e_2 \overline{X}e_2= 0$. 
\end{thm}
\vspace{0.2cm} 		

Now we will discuss the arrangement of this article. In \S 2, We gather all of the materials required for this article. 
We recall, in particular, some fundamental facts concerning states defined on a von Neumann algebra and non-commutative $L^1$-spaces associated with a faithful normal semifinite trace.  
In the same section we also discuss the stochastic and bilateral almost uniform convergence.
\S 3  and \S 4 are  devoted to study  the Neveu decomposition. In \S 5, we discuss some examples. 
In \S 6 we discuss and recall pointwise ergodic theorems.  In the final section, we combine the previous section's conclusions to show the stochastic ergodic theorem.

	\section{Preliminaries}
		
\subsection{von Neumann algebra and linear functional} 
In this article, $\CH$ will denote a separable Hilbert space  and $M (\subseteq \CB(\CH))$ will represent a von Neumann algebra. $M$ possesses several locally convex topologies along with the norm topology induced from $\CB(\CH)$. In this paper, $\norm{\cdot}$ will denote the norm on $M$ as well as on $\CB(\CH)$. $ M_+$ will denote the set of all positive elements of $M$. For the definition of these topologies and various other facts regarding von Neumann algebras, we refer to \cite[chapter 1]{Stratila1979}. 

Write $M_*$ for the set of all $w$-continuous linear functionals on $M$. It is a norm closed subspace of $M^*$, the Banach space dual of $M$. The dual space norm on $M^*$ and its restriction to $M_*$ will be denoted by $\norm{\cdot}_1$ in the sequel. Then  $M$ is isomorphic to $ (M_*)^*$ via a canonical bi-linear form defined on $M \times M_*$. For the proof of this fact, we again refer to \cite[Lemma 1.9]{Stratila1979}. Under this identification, $M_*$ is called the predual of $M$.

A linear functional $\varphi$ on $M$ is called positive if $\varphi(x^*x) \geq 0$ for all $x \in M$ and we refer it by $\varphi \geq 0$.  A  positive linear functional $\varphi$ is called  state if $\varphi(1)=1$. The  elements of $M_*$ will be called   normal linear functionals. A self-adjoint linear functional is a finite linear combination of positive linear functionals. The set of all normal positive  linear functionals will be denoted by $M_{* +}$ and self-adjoint elements of $M_*$ will be denoted by $M_{* s}$.

The support of a self-adjoint operator  $x$ in $M$, henceforth denoted by $s(x)$, is the smallest projection in $M$ such that $s(x)x=x$, thus, equivalently, $xs(x)=x$. On the other hand, for a positive normal linear functional $\varphi$, the set $\{ e \in M: e \text{ is a projection and } \varphi(e)=0 \}$ is increasingly directed and if the projection  $p \in M$ be its least upper bound, then one can show that $\varphi(p)=0$. Then the projection $1-p$ is called the support of $\varphi$ and it will be denoted by $s(\varphi)$ in the sequel. A positive linear functional $\varphi$ is said to be faithful if for all $x \in M_+$ with $\varphi(x)=0$ implies $x=0$. It is also well-known that a positive normal linear functional $\varphi$ is faithful on $s(\varphi) M s(\varphi)$. For more details we refer to \cite{Stratila1979}. We abbreviate the faithful normal  linear functional on $M$  as  f.n linear functional. 
The collection of all (resp. non-zero) projections  in a von Neumann algebra $M$ will be denoted by $\CP(M)$ (resp. $\CP_0(M)$). Let $p, q \in \CP(M)$, if $pq= 0$, i.e, $p$ and $q$ are orthogonal then we write $p\perp q$.

We now briefly recall singular linear functionals on $M$ and describe the decomposition of a positive linear functional into normal and singular ones. The following materials regarding  singular linear functionals are taken from \cite{Takesaki1958, Takesaki2002}.
\begin{defn}\cite{Takesaki1958}
	A positive linear functional $\varphi$ on a von Neumann algebra $M$ is said to be singular if there exists a positive normal linear functional $\psi$ on $M$ such that $(\varphi - \psi) \geq 0$, then $\psi=0$.
\end{defn}
We will frequently use the following characterisation of a singular linear functional on $M$.

\begin{thm}\cite[Theorem 3.8, pp. 134]{Takesaki2002}\label{defining property of singular functional}
	A positive linear functional $\varphi$ on $M$ is singular if and only if for any non-zero projection $e \in M$, there exists a non-zero projection $f \leq e$ in $M$ such that $\varphi(f)=0$.
\end{thm}

\begin{thm}\cite[Theorem 3]{Takesaki1958}\label{takesaki decomposition}
	Let $\varphi$ be a positive linear functional on $M$. Then there exist unique normal linear functional $\varphi_{n}$ and a singular linear functional $\varphi_{s}$ on $M$ such that 
	\begin{align*}
	\varphi= \varphi_n + \varphi_s .
	\end{align*}
\end{thm}

\subsection{Non-commutative $L^p$-spaces} 
Let $M \subseteq \CB(\CH)$ be a   von Neumann algebra. Recall that 
a closed densely defined operator $X: \mathcal{D}(X)\subseteq \CH \to \CH$  is called  affiliated to $M$ if $X$ commutes with every unitary operator $u'$ in $M'$, where $M'$ is  the commutant of $ M$ in $\CB(\CH)$. From now on if $X$ is affiliated to $M$, we will denote it by $X \eta M$.

Now we put up  a brief account on non-commutative $L^p$-spaces for a semifinite von Neumann algebra $N$ with a  faithful,  normal,  semifinite trace $\tau$. For more details, we refer to \cite{Hiai2020}. We abbreviate faithful, normal, and semifinite as f.n.s for future reference.  Let $L^2(N, \tau )$ be the GNS Hilbert space. Write $\CH = L^2(N, \tau )$ and we identify $N$ as a von Neumann subalgebra of $\CB(\CH)$. 
\begin{defn}
	An operator $X$ (possibly unbounded), defined on $\CH$, is said to be $\tau$-measurable if for every $\epsilon >0$, there is a projection $e$ in $N$ such that $e \CH \subseteq \mathcal{D}(X)$ and $\tau(1-e)< \epsilon$. The set of all closed, densely defined operators affiliated to $N$ which are $\tau$-measurable is denoted by $L^0(N, \tau)$.
\end{defn}

\begin{rem}\label{mble ops are star algebra}
	Let $X, Y \in L^0(N, \tau)$, and then it is easy to see that $X+Y$ and $XY$ are densely defined and closable. Moreover, $L^0(N, \tau)$ is a $*$-algebra with respect to adjoint operation $*$, the strong sum $\overline{X+Y}$, and the strong product $\overline{XY}$, where $\overline{(\cdot)}$ denotes the closure of an operator. 
\end{rem}

\noindent For all $\epsilon, \delta >0$, let us consider the following set,
\begin{align*}\label{nbd basis for measure topology} 
\mathcal{N}(\epsilon, \delta):= 
\lbrace X \in L^0(N, \tau): \exists \text{ projection } e \in N \text{ such that } \  \norm{Xe} \leq \epsilon \text{ and } \tau(1-e) \leq \delta \rbrace.
\end{align*}
Note that the collection of sets $\{ X +\mathcal{N}(\epsilon, \delta) : X \in L^0(N, \tau), \epsilon>0, \delta>0 \}$ forms a neighborhood basis on $L^0(N, \tau)$.

\begin{defn}
	The topology  generated by the neighborhood basis $\{ X +\mathcal{N}(\epsilon, \delta) : X \in L^0(N, \tau), \epsilon>0, \delta>0 \}$ is  called the measure topology on $L^0(N, \tau)$. 
\end{defn}
\begin{defn}	
	A sequence $\{ X_n \}_{n \in \N}$ in $L^0(N, \tau)$ is said to converge in measure (or converge stochastically ) to $X \in L^0(N, \tau)$ if for all $\epsilon, \delta>0$ there exists a sequence of projections $\{e_n\}_{n \in \N}$ in $N$ and $n_0 \in \N$ such that for all $n \geq n_0$
	\begin{align*}
	\tau(1-e_n) < \delta  \text{ and } \norm{(X_n-X) e_n} < \epsilon.
	\end{align*}
\end{defn}

\begin{rem}\label{measure topo mble ops}
	We note that with respect to measure topology, $L^0(N, \tau)$ becomes a complete, metrizable, Hausdorff space. Moreover, $N$ is dense in $L^0(N, \tau)$ in this topology [cf. \cite{Hiai2020}, Theorem 4.12]. 
\end{rem}

The following theorem from \cite{Chilin2005} provides an equivalent condition for convergence of a sequence in measure. Henceforth, we use the criteria in the following theorem (Theorem \ref{equiv of bm and N}) as the definition of convergence in measure.

\begin{thm}[Theorem 2.2, \cite{Chilin2005}]\label{equiv of bm and N}
	A sequence of operators $\lbrace X_n \rbrace_{n \in \N} \subseteq L^0(N, \tau)$ converges in measure to $X \in L^0(N, \tau)$ iff for all $\epsilon > 0$ and $\delta > 0$ there exists $n_0 \in \N$ and a sequence of projections $\lbrace e_n \rbrace_{n \in \N}$ in $N$ such that for all $n \geq n_0$,
	\begin{align*}
	\tau(1-e_n) < \delta \text{ and } \norm{e_n (X_n-X) e_n} < \epsilon.
	\end{align*}
\end{thm}

We now define another notion of convergence of sequences in $L^0(N, \tau)$, which will be helpful in our context.

\begin{defn}
	A sequence of operators $\lbrace X_n \rbrace_{n \in \N} \subseteq L^0(N, \tau)$ converges bilaterally almost uniformly (b.a.u) to $X \in L^0(N, \tau)$ if for all $\epsilon > 0$  there exists a projection $e \in N$ with $\tau(1-e) < \epsilon$ such that for all $\delta > 0$ there exists $n_0 \in \N$ such that for all $n \geq n_0$,
	\begin{align*}
	 \norm{e(X_n-X) e} < \delta. 
	\end{align*}
\end{defn}

\begin{rem}\label{a.u. stronger than m}
	We immediately notice  that	the bilateral almost uniform convergence of a sequence implies the convergence of the sequence in measure.
\end{rem}

We put down the following proposition for our future reference. The proof is simple and hence we omit it.
\begin{prop}\label{bau convergence closed under addition and mult}
	Suppose $\{X_n \}_{n \in \N}$ and $\{ Y_n \}_{n \in \N}$ are two sequences in $L^0(N, \tau)$ such that $\{X_n \}_{n \in \N}$ converges in measure (resp. b.a.u) to $X$ and $\{ Y_n \}_{n \in \N}$ converges in measure (resp. b.a.u) to $Y$. Then, for all  $c \in \C$,  $\{ c X_n + Y_n \}_{n \in \N}$ converges in measure (resp. b.a.u) to $cX + Y$.
\end{prop}

Let $L^0(N, \tau)_+$ denote the positive cone of $L^0(N, \tau)$. Extend the trace $\tau$ on $N_+$ to $L^0(N, \tau)_+$ and denote it again by $\tau$ by slight abuse of notation.  For any $X \in L^0(N, \tau)_+$, it is defined by 
\begin{align*}
\tau(X) := \int_{0}^{\infty} \lambda d \tau(e_{\lambda}),
\end{align*}
where $X= \int_{0}^{\infty} \lambda d(e_{\lambda})$ is the spectral decomposition.

Although, for  this article we need only the description of $L^p(N, \tau)$ for $p = 1$ or $ \infty$, but we define it for all $ 0 < p \leq \infty$ since the definitions are similar for all $p$. 
\begin{defn}
	For $0< p \leq \infty$, the non-commutative $L^p$-space on $(N, \tau)$ is defined by 
	\[ L^p(N,\tau):= 
	\begin{cases}
	\lbrace X \in L^0(N, \tau) : \norm{X}_p:= \tau(\abs{X}^p)^{1/p} < \infty \rbrace  & \text{ for } p \neq \infty, \\
	(N, \norm{\cdot}) & \text{ for } p= \infty
	\end{cases}
	\]
	where, $\abs{X}= (X^*X)^{1/2}$. 
\end{defn}

The positive cone of $L^p(N, \tau)$ will be denoted by $L^p(N, \tau)_+$. Now we list a few essential properties of $L^p(N,\tau)$ without proofs. We will use these properties recursively in the sequel. For the proofs, we refer to \cite{Hiai2020}.

\begin{thm}\label{basic thm in Lp}
	\begin{enumerate}
		\item[(i)] For all $1 \leq p \leq \infty$, $L^p(N,\tau)$ is a Banach space with respect to the norm $\norm{\cdot}_p$. Moreover, for all $1<p<\infty$, $L^p(N,\tau)$ is reflexive.
		\item[(ii)] Let $1 < p < \infty$ and $1/p + 1/q =1$. Then the map $\Psi: L^q(N, \tau) \to L^p(N, \tau) $ defined by $\Psi(b)(a)= \tau(ab)$ for $b \in L^q(N, \tau), a \in L^p(N, \tau)$ is a surjective linear isometry.
		\item[(iii)] The map $\Psi: L^1(N, \tau) \to N_* $ defined by $\Psi(X)(a)= \tau(Xa)$ for $X \in L^1(N, \tau), a \in N$ is a surjective linear isometry which preserves the positive cones.
	\end{enumerate}
\end{thm}

The following discussion is also going to play an important role in the sequel. Let $N$ be a von Neumann algebra acting on a Hilbert space $\CH$ and equipped with a faithful,  normal,  semifinite trace $\tau$ and $e$ be a projection in $N$. Further, assume that $N_e:= \{ x_e:= ex_{\upharpoonleft e\CH} : x \in N \}$ denote the reduced von Neumann algebra. Define the reduced trace on $N_e$ as
\begin{align*} 
\tau_e(x_e)= \tau(exe), \text{ for all } x \in N.
\end{align*}

\begin{rem}
	\begin{enumerate}
		\item[(i)] Since $\tau$ is a faithful, normal, semifinite trace, $\tau_e$ also has similar properties.
		\item[(ii)] Let $1\leq  p \leq \infty$ and $X \in L^p(N_e, \tau_e)$. Define $\tilde{X}$ on $\CH$ by
		\begin{align*}
		\tilde{X}\xi = Xe \xi, \text{ for all } \xi \in \mathcal{D}(\tilde{X}):= \mathcal{D}(X) \oplus (1-e)\CH.
		\end{align*}
		Then the mapping $L^p(N_e, \tau_e) \ni X \mapsto \tilde{X} \in eL^p(N, \tau)e$ defines an isomorphism as Banach space for $1 \leq p \leq \infty$.  From now onwards we identify $L^p(N_e, \tau_e)$ with $eL^p(N ,\tau)e$.
		\item[(iii)] It follows from  \cite[Theorem 4.12 and Lemma 5.3]{Hiai2020} that $e L^p(N, \tau) e \subseteq L^p(N, \tau)$.
	\end{enumerate}
\end{rem}
\subsection{Actions by amenable semigroups}

	Throughout this article, unless otherwise mentioned, $G$ will denote a  locally compact second countable Hausdorff (shortly read it as LCSH)  semigroup and $m$ be a $\sigma$-finite measure on $G$, which is both left and right invariant (i.e. $m(uB)= m(B)$ and $m(Bu)= m(B)$ for all $u \in G$).  In the sequel $\K$ will always denote either $\Z_+$ or $\R_+$. We consider a collection $\{K_l\}_{l \in \K}$ of measurable subsets of $G$ having the following properties. 
	\begin{enumerate}
		\item[(P1)] $0< m(K_l) < \infty$ for all $l \in \K$.
		\item[(P2)] $\lim_{l \to \infty} \frac{m(K_l \Delta K_l g)}{m(K_l)}=0$ and $\lim_{l \to \infty} \frac{m(K_l \Delta gK_l)}{m(K_l)}=0$ for all $g \in G$.
	\end{enumerate}
Note that  such a  net  is called F\o lner net and in that  case $G$ is referred as amenable semigroup. Thus, when we say $G$ is a amenable semigroup, we mean G is  is LCSH semigroup having F\o lner net with respect to fix invariant measure. In this article, $G$ will be meant for LCSH amenable semigroup unless otherwise mentioned.
\begin{rem}
	We like to pointout that when $G$ is a group then we don't need to assume  the Haar measure to be left-right invariant. Only when it is proper semigroup (i.e, it is not closed under inverse), we require  for certain results in the sequel that the Haar measure on $G$ is  left-right invariant.  
\end{rem}

Now we quickly recall ordered Banach space. A  real Banach space $E$ paired with a closed convex subset $K$ satisfying $K \cap -K = \{0\}$ and $\lambda K \subseteq K$ for $\lambda \geq 0$ induces a partial order on $E$ given by $x \leq y$ if and only if $y-x \in K$ where $x,y \in E$. With such a partial order, $E$ will be referred to as an ordered Banach space. In our context, it is easy to verify that $M, M^*, M_*$, $L^p(M, \tau)$ for $1 \leq p \leq \infty$ become ordered  Banach spaces with respect to the natural order. Let us now define an action of $G$ on ordered Banach spaces.
	
	\begin{defn}\label{action on Banach sp}
		Let $E$ be an Banach space. A map $\Lambda$ defined by 
		\begin{align*}
		G \ni g \xrightarrow[]{\Lambda} \Lambda_g \in \CB(E)
		\end{align*}
		is called an action if $\Lambda_g \circ \Lambda_h= \Lambda_{gh}$ for all $g,h \in G$. It is called anti-action if $\Lambda_g \circ \Lambda_h= \Lambda_{hg}$ for all $g,h \in G$. In this article, we consider both actions and anti-actions $\Lambda= \{\Lambda_g\}_{g \in G}$ which satisfy the following conditions.
		\begin{enumerate}
			\item[(C)] For all $x \in E$, the map $g \rightarrow \Lambda_g(x)$ from $G$ to $E$ is continuous. Here we take $w^*$-topology when $E=M$ and norm topology otherwise.
			\item[(UB)] $\sup_{g \in G} \norm{\Lambda_g} \leq 1$.
			\item[(P)] Suppose $x \in E$ with $x \geq 0$, then  $\Lambda_g(x) \geq 0$ for all  $g \in G$.  
				\item[(U)] When $E = M$, we assume that $ \Lambda_g(1) \leq  1$ for all  $g \in G$. 
		\end{enumerate}
		We refer the triple $(E, G, \Lambda)$  as a non-commutative dynamical system.  
	\end{defn}
	Let $(M, \Lambda, G)$ be a non-commutative dynamical system and $\varphi \in M^*$, it is called $G$-invariant if $\varphi(\Lambda_g(x)) = \varphi(x)$ for all $x \in M$ and $g \in G$.\\\\
	Now let $(E, G, \Lambda )$ be a  non-commutative dynamical system and $\{K_n\}_{n \in \N}$ be a \textit{F\o lner sequence} in $G$. For all $x \in E$, we consider the following average 
	\begin{align}\label{Average}
	A_n(x):= \frac{1}{m(K_n)} \int_{K_n} \Lambda_g(x) dm(g).
	\end{align}
	We note that for $x \in E$, the map  $G \ni g \rightarrow \Lambda_g(x) \in E $ is continuous in norm of $E$.
	Further, when $E=M$, then it is  $w^*$-continuous. Therefore, in both cases the integration in eq. \ref{Average} is well defined. In addition 
	by \cite[Proposition 1.2, pp-238]{Takesaki2003},  for all $n \in \N$, $\varphi \in M_*$ and $x \in M$ the following holds
	\begin{align*}
	\varphi(A_n(x)) = \frac{1}{m(K_n)} \int_{K_n} \varphi(\Lambda_g(x)) dm(g). 
	\end{align*}\\
	Now for all $g \in G$, consider
	\begin{align}\label{action on M*}
	\Lambda_g^* : M^* \to M^* \text{ by } \Lambda_g^*(\varphi)(x)= \varphi(\Lambda_g(x)) \text{ for all } \varphi \in M^*, x \in M.
	\end{align}
	For all $n \in \N$, we also consider the following average defined by
	\begin{align}\label{averaging operator in M*}
	A_n^*: M_* \to M_*; \		\varphi \mapsto A_n^*(\varphi)(\cdot):= \varphi(A_n(\cdot))= \frac{1}{m(K_n)} \int_{K_n} \Lambda_g^*(\varphi)(\cdot) dm(g).
	\end{align}
	We note that for all $n \in \N$,  $\norm{A_n} \leq \sup_{g \in G} \norm{\Lambda_g}$ and  $\norm{A_n^*} \leq \sup_{g \in G} \norm{\Lambda_g}$. These will be called averaging operators for future references. We make no difference between these two averaging operators $A_n$ and $A_n^*$ for the convenience of notation, unless and otherwise, it is not clear from the context.

The following discussion regarding dual and predual map  will be used in future. 	
Let $M$ be a semifinite von Neumann algebra with a f.n semifinite trace $\tau$, then we identify $M_* $ with $ L^1(M, \tau)$ . Suppose  $T : M \rar M $ be a  contractive positive normal map and $T^*$ and $T_*$ be the dual and predual operator  of $T$ respectively  on $ M^*$ and $M_*$. Then  note that $ T^*|_{ M_*} = T_* $. \\\\
We write the following version of mean ergodic theorem, the proof may be  folklore in the literature. But for the completeness we add a proof. This version of mean ergodic theorem will be used in the sequel. Indeed, it will be used to prove the existence of weakly wandering operators.   
\begin{thm}[Mean Ergodic Theorem]\label{mean erg thm}
		Let $(E, \norm{\cdot})$ be a Banach space and $( E, G, \alpha)$ be a non-commutative dynamical system.   Consider the associated ergodic averages
		\begin{align*}
			A_n (\xi) := \frac{1}{m(K_n)} \int_{K_n} \alpha_g (\xi) dm(g), ~ \xi \in E.
		\end{align*}
		Suppose $ E^G= \{ x \in E: \alpha_g(x) =x  \text{ for all } g \in G \}$. Then the following are equivalent.
		\begin{enumerate}
			\item For all $\xi \in E$ there exists $\overline{\xi} \in E^G$ such that $\lim_{n \to \infty} A_n \xi= \overline{\xi}$.
			\item For all $\xi \in E$ there exists $\overline{\xi} \in E^G$ and a subsequence $(n_k)$ such that weak-$\lim_{k \to \infty} A_{n_k} \xi= \overline{\xi}$.
			\item For all $\xi \in E$ there exists $\overline{\xi} \in E^G \cap \overline{co}^{weak} \{\alpha_g \xi : g \in G \}$.
			\item For all $\xi \in E$ there exists $\overline{\xi} \in E^G \cap \overline{co}^{\norm{\cdot}} \{\alpha_g \xi : g \in G \}$.
		\end{enumerate}
	\end{thm}

	%
	%
	
	\begin{proof}
		The proof of $(1) \Rightarrow (2)$ is clear. Proof of $(3) \Leftrightarrow (4)$ implies from the Mazur's theorem which states that any convex subset of $E$ has same closure with respect to norm and weak topology.\\
		
		$(2) \Rightarrow (3): $ Let $\xi \in E$. By hypothesis, there is $\overline{\xi} \in E^G$ and a subsequence $(n_k)$ such that weak-$\lim_{k \to \infty} A_{n_k} \xi= \overline{\xi}$. We claim that $\overline{\xi} \in \overline{co}^{weak} \{\alpha_g \xi : g \in G \}$. For this it is enough to show that $\overline{\xi} \in \overline{co}^{\norm{\cdot}} \{\alpha_g \xi : g \in G \}$. Suppose it is not true, then by Hahn- Banach separation theorem there exists $\Lambda \in E^*$ and $a>0$ such that
		\begin{align*}
			Re \Lambda (\overline{\xi}) \geq a + Re \Lambda(f) \text{ for all } f \in \overline{co}^{\norm{\cdot}} \{\alpha_g \xi : g \in G \}.
		\end{align*}
		In particular, 
		\begin{align*}
			Re \Lambda (\overline{\xi}) \geq a + Re \Lambda(\alpha_g(\xi)) \text{ for all } g \in G.
		\end{align*}
		Therefore, 
		\begin{align*}
			a + Re \Lambda(A_{n_k}(\xi))
			&= a + \frac{1}{m(K_{n_k})} \int_{K_{n_k}} Re \Lambda(\alpha_g(\xi)) dm(g) \\
			&\leq \frac{1}{m(K_{n_k})} \int_{K_{n_k}} Re \Lambda (\overline{\xi}) dm(g) \\
			&= Re \Lambda (\overline{\xi}) \text{ for all } k \in \N.
		\end{align*}
		Now passing limit as $k \to \infty$, we obtain $a+ Re \Lambda (\overline{\xi}) \leq Re \Lambda (\overline{\xi})$, which is a contradiction. \\
		
		$(4) \Rightarrow (1): $ Let $\xi \in E$ and $\epsilon>0$.  We first find a convex combination $\xi':= \sum_{1}^{m} \lambda_i \alpha_{g_i}(\xi)$, where $\sum_{1}^{m} \lambda_i = 1$ such that $\norm{\overline{\xi}- \xi'}< \epsilon$. Also for all $n \in \N$, note that 
		\begin{align*}
			A_n(\xi) - A_n(\xi')&= \sum_{1}^{m} \lambda_i  (A_n(\xi) - A_n(\alpha_{g_i}(\xi)))\\
			&= \sum_{1}^{m} \lambda_i \frac{1}{m(K_n)}\Big[\int_{K_n} \alpha_h(\xi) dm(h) - \int_{K_n} \alpha_{hg_i} (\xi) dm(h) \Big]\\
			&=  \sum_{1}^{m} \lambda_i \frac{1}{m(K_n)} \Big[\int_{K_n} \alpha_h(\xi) dm(h) - \int_{K_n g_i} \alpha_{h} (\xi) dm(h) \Big].
		\end{align*}
		Therefore, for all $n \in \N$ we have
		\begin{align*}
			\norm{A_n(\xi) - A_n(\xi')} \leq C \norm{\xi} \sum_{1}^{m} \lambda_i\frac{m(K_n g_i \Delta K_n)}{m(K_n)}.
		\end{align*}
		Now by the F\o lner condition, we choose $n_0 \in \N$ such that $\norm{A_n(\xi) - A_n(\xi')} \leq C \norm{\xi} \epsilon$ for all $n \geq n_0$. Now since $A_n(\overline{\xi})= \overline{\xi}$ for all $n \in \N$, we have
		\begin{align*}
			\norm{A_n(\xi) - \overline{\xi}} &\leq \norm{A_n(\xi) - A_n(\xi')} + \norm{A_n(\xi' - \overline{\xi})}\\
			& \leq C\norm{\xi} \epsilon + \norm{A_n} \norm{\xi' - \overline{\xi}}\\
			& \leq C \epsilon (\norm{\xi} + 1) \text{ for all } n \geq n_0.
		\end{align*}
		This completes the proof.
	\end{proof}

	\section{Neveu Decomposition}
	
Let $G$ be a amenable semigroup and $M$ be a von Neumann algebra. Suppose $(M, G, \alpha)$ is a non-commutative dynamical system.  In the beginning of this section we discuss the existence of invariant states for $(M, G, \alpha)$. 
Later part of this section we assume that $M$ is semifinite von Neumann algebra with a f.n.s trace $\tau$ and then  study   the existence of weakly wandering operator for $(M, G, \alpha)$ and we obtain the Neveu  decomposition for $(M, G, \alpha)$  and thus this completely settle the problem of Neveu decomposition for the  action of amenable semigroup on semifinite von Neumann algebras. We begin with the following proposition. 
	
	\begin{prop}\label{existence of inv normal state}
		Let $M$ be a von Neumann algebra equipped with a f.n  state $\varphi$ and $(M, G, \alpha)$ be a non-commutative dynamical system. Further, for $ e\in \CP_0(M)$,  assume that the following holds.
		\begin{align}\label{inf condition-1}
			\inf_{n \in \N} A_n(\varphi)(p)>0, ~  \text{ for all }  p \in \CP_0(M) \text{ with  }   p \leq e. 
		\end{align}
		Then, there exists an invariant normal state $\nu_{\varphi}$ with $s(\nu_{\varphi}) \geq e$.
	\end{prop}
	
	\begin{proof}
		First note that since $\norm{A_n(\varphi)} \leq \norm{\varphi}$ for all $n \in \N$, by Banach-Alaglou theorem applied on $M^*$, we obtain a subsequence $\{A_{n_k}(\varphi)\}_{k \in \N} \subseteq M^*_1$, which converges pointwise. We define 
		\begin{align*}
			\overline{\varphi}(x):= \lim_{k \to \infty} A_{n_k}(\varphi)(x), x \in M.
		\end{align*}
		We claim that $\overline{\varphi} \circ \alpha_h= \overline{\varphi}$ for all $h \in G$. Indeed, for all $h \in G$, $x \in M$ and $k \in \N$ we have
		\begin{align*}
			\abs{\overline{\varphi}(\alpha_h(x))- \overline{\varphi}(x)} \leq  & \abs{\overline{\varphi}(\alpha_h(x)) - A_{n_k}(\varphi)(\alpha_h(x))} +  \\ 
			& \abs{ A_{n_k}(\varphi)(\alpha_h(x)) - A_{n_k}(\varphi)(x)} + \abs{A_{n_k}(\varphi)(x) - \overline{\varphi}(x)}\\
			\leq  & \abs{\overline{\varphi}(\alpha_h(x)) - A_{n_k}(\varphi)(\alpha_h(x))} +  \\ 
			& \frac{m(K_{n_k} h \Delta K_{n_k})}{m(K_{n_k})}  + \abs{A_{n_k}(\varphi)(x) - \overline{\varphi}(x)}.
		\end{align*}
		Hence by F\o lner condition and the definition of $\overline{\varphi}$, the right hand side of the above equation converges to $0$. Therefore, we conclude that $\alpha_h^*(\overline{\varphi})= \overline{\varphi}$ for all $h \in G$.
		It is clear that since $\varphi$ is a state, $\overline{\varphi}$ is also a positive linear functional and $\overline{\varphi}(1)=1$. Let 
		\begin{align*}
			\overline{\varphi}= \overline{\varphi}_n + \overline{\varphi}_s
		\end{align*}
		be the decomposition of $\overline{\varphi}$ in accordance with \cite[Theorem 3]{Takesaki1958}, where $\overline{\varphi}_n$ is a normal linear functional and $\overline{\varphi}_s$ is a positive linear functional which is singular. Therefore, we have, $\overline{\varphi} \circ \alpha_g = \overline{\varphi}_n \circ \alpha_g + \overline{\varphi}_s \circ \alpha_g$ for all $g \in G$.
		
		Now fix a $g \in G$ and further decomposing $\overline{\varphi}_s \circ \alpha_g$ in normal and singular component  and  we obtain
		\begin{align}\label{n-s component eq1}
			\overline{\varphi}_n + \overline{\varphi}_s= \overline{\varphi} =\overline{\varphi} \circ \alpha_g = \overline{\varphi}_n \circ \alpha_g + (\overline{\varphi}_s \circ \alpha_g)_n + (\overline{\varphi}_s \circ \alpha_g)_s,
		\end{align}
		which implies
		\begin{align*}
			\overline{\varphi}_n - \overline{\varphi}_n \circ \alpha_g - (\overline{\varphi}_s \circ \alpha_g)_n= (\overline{\varphi}_s \circ \alpha_g)_s - \overline{\varphi}_s= 0.
		\end{align*}
		Hence, first observe that $(\overline{\varphi}_s \circ \alpha_g)_s = \overline{\varphi}_s$ and 	$\overline{\varphi}_n - \overline{\varphi}_n \circ \alpha_g - (\overline{\varphi}_s \circ \alpha_g)_n =0$. Now we wish to show that $$(\overline{\varphi}_s \circ \alpha_g)_n =0.$$
		Indeed, observe the following  
		\begin{align*}
		&  \overline{\varphi}_s = (\overline{\varphi}_s \circ \alpha_g)_s\\
		\implies& \overline{\varphi}_s  +(\overline{\varphi}_s \circ \alpha_g)_n  = (\overline{\varphi}_s \circ \alpha_g)_s +(\overline{\varphi}_s \circ \alpha_g)_n = \overline{\varphi}_s \circ \alpha_g\\
		\implies &  \overline{\varphi}_s(1)  +(\overline{\varphi}_s \circ \alpha_g)_n(1) = \overline{\varphi}_s \circ \alpha_g(1) \leq \overline{\varphi}_s (1), \text{ as } \alpha_g(1) \leq 1\\
		\implies & (\overline{\varphi}_s \circ \alpha_g)_n(1)\leq 0. 
		\end{align*}
As $(\overline{\varphi}_s \circ \alpha_g)_n $ is positive  and $(\overline{\varphi}_s \circ \alpha_g)_n(1)\leq 0$, so, $(\overline{\varphi}_s \circ \alpha_g)_n =0$.	Thus we have 	
$$ \overline{\varphi}_n = \overline{\varphi}_n \circ \alpha_g \text{ for all } g \in G.$$		

		Therefore, $\overline{\varphi}_n$ is a normal linear functional which is $G$-invariant. We define, $\nu_{\varphi}:= \frac{1}{\overline{\varphi}_n(1)}\overline{\varphi}_n$.
		
		To show $s(\nu_{\varphi}) \geq e$, we first let $p$ be any non-zero subprojection of $e$ in $M$. Then by \cite[Theorem 3.8, pp. 134]{Takesaki2002}, observe that there exists a non-zero subprojection $p'$ of $p$ in $M$ such that $\overline{\varphi}_s(p')=0$. Therefore, we have
		\begin{align*}
			\nu_{\varphi}(p) \geq \nu_{\varphi}(p')=  \frac{1}{\overline{\varphi}_n(1)} \overline{\varphi}(p')= \frac{1}{\overline{\varphi}_n(1)} \lim_{k \to \infty} A_{n_k}(\varphi)(p') \geq  \frac{1}{\overline{\varphi}_n(1)} \inf_{n \in \N} A_n(\varphi)(p')>0.
		\end{align*}
		This completes the proof.
	\end{proof}
	
	The following theorem characterizes to find a maximal invariant state in term of its support  projection satisfying the condition in eq. \ref{inf condition-1}. 
	
	\begin{thm}\label{supp of maximal state}
		Let $(M, G, \alpha)$ be a non-commutative dynamical system with a with a f.n  state $\varphi \in M_*$ . Then for $ e\in \CP_0(M)$, the following statements are equivalent.
		\begin{enumerate}
			\item There exists an invariant normal state $\rho$ on $M$ with $s(\rho)= e$ such that,
			if $\nu$ is any invariant normal state  on $M$, then   $s(\nu) \leq e$.
			\item $e$ is the maximal projection satisfying the following condition:
			\begin{align}\label{inf condition-2}
				\inf_{n \in \N} A_n(\varphi)(p)>0, ~  \text{ for all }  p \in \CP_0(M) \text{ with  }   p \leq e. 
			\end{align} 
		\end{enumerate}
	\end{thm}
	
Before proving this theorem we recall the following proposition without a proof. The proof of  the proposition is straightforward, but reader may also look at  \cite[Proposition 3.4]{Bik-Dip-neveu}.	
	
	\begin{prop}\label{inf condition-3}
		Let $(M, G, \alpha)$ be a non-commutative dynamical system with a f.n  state $\varphi \in M_*$.  If there exists a  $\rho \in M_{* +}$ such that $\alpha_g^*(\rho)= \rho$, then for any  $x \in M_+$ with $\rho(x) \neq 0$, we have $\inf_{g \in G} \alpha_g^*(\varphi)(x) >0$.
	\end{prop}

	\begin{proof}[Proof of Theorem \ref{supp of maximal state}]
		\emph{(1) $\Rightarrow$ (2):} Since $e$ is the support of the invariant state $\rho$, we have $\rho(p)>0$ for all projection $0 \neq p \leq e$ in $M$. Therefore, by virtue of the  Proposition \ref{inf condition-3} we conclude that 
		\begin{align*}
			\inf_{n \in \N} A_n(\varphi)(p)>0, \text{ for all }   p \in \mathcal{P}_0(M) \text{ with } p \leq e.
		\end{align*} 
		Now suppose that $e$ is not the maximal projection that satisfy the condition in eq. \ref{inf condition-2}. Therefore, there exists a non-zero  $f \in \CP(M)$  which is not a sub-projection of $e$ but satisfies the condition in eq. \ref{inf condition-2}. Therefore, by Proposition \ref{existence of inv normal state} there exists an invariant normal state $\nu_{\varphi}$ on $M$ such that $s(\nu_{\varphi}) \geq f$. This is a contradiction to the hypothesis.
		
		\noindent \emph{(2) $\Rightarrow$ (1):} Since the non-zero projection $e$ satisfies the condition in eq. \ref{inf condition-2}, by virtue of Proposition \ref{existence of inv normal state}, there exists an invariant normal state  $\nu_{\varphi}$ on $M$ such that $s(\nu_{\varphi}) \geq e$. Again by applying Proposition \ref{inf condition-3}, one can show that $s(\nu_{\varphi})$ satisfies the condition in eq. \ref{inf condition-2}. Since $e$ is the maximal projection satisfying the condition in eq. \ref{inf condition-2}, we obtain $s(\nu_{\varphi}) \leq e$. Similarly it will also follow that $s(\nu)\leq e$ for any invariant normal state $\nu$ on $M$.
	\end{proof}

\begin{defn}\label{weakly wandering defn}
	Let $(M,  G, \alpha )$ be a covariant system and $x$ be a positive operator in $M$. Then $x$ is said to be a weakly wandering operator if 
	\begin{align*}
	\lim_{n \to \infty} \norm{A_nx} =0.
	\end{align*}
\end{defn}	
	Suppose $( M, G, \alpha)$  is  a non-commutative dynamical system and $e \in \CP_0(M)$ such that if $\varphi$ is any $G$-invariant state then $s(\varphi) \leq e$. Then we wish to show the existence of weakly wandering projection  $ q \in  \CP_0(M)$ with $q \leq 1-e $, i.e $ A_n(q) \xrightarrow{ n \rar \infty } 0 $ in $\norm{\cdot}$. 
For this purpose we assume that $M$ is semifinite von  Neumann algebra with a f.n semifinite trace $\tau$. Further, $p \in \CP_0(M) $, denote the reduced von Neumann algebra $pMp$ by $M_p$, i.e, $M_p = pMp$.  We like to emphasise that the following result is key to find the weakly wandering operators.

	\begin{lem}\label{main lem-semifinite2}
		Let $(M,\tau)$ be a semifinite von Neumann algebra and $r \in \CP_0(M)$.  If $K$ is a weak*-compact subset of $(M^*_1)_+$ such that $\mu_{\upharpoonleft M_{r}}$ is singular for all $\mu \in K$, then for all $f \in \CP_0(M_{r})$ there exists a $p \in \CP_0(M)$ with $p \leq f$ such that $\mu(p)=0$ for all $\mu \in K$.
	\end{lem}
	
	
	\begin{proof}
		Let $f \in \CP_0(M_{r})$.
		Then consider a $p \in \CP_0(M)$ with $p \leq f$ such that $\tau(p)< \infty$. Define a faithful, normal state $\tau_p$ on $M$ by $\tau_p(x)= \frac{\tau(pxp)}{\tau(p)}$. Note that $\tau_p(p)= \tau_p(1)=1$. Let $0 < \epsilon < \frac{1}{2}$ and define $\Phi_\epsilon:= \{0 \leq x \leq p: \tau_p(x) \geq 1-\epsilon\}$.  If $\mu(p)=0$ for all $\mu \in K$, then we are done.
		
		Let $\mu \in K$. Since $\mu$ is singular there exists a sequence of projections $\{p_n\}$ in $M$ such that $p_n \uparrow p$ and $\mu(p_n)=0$. Choose $n \in \N$ such that $\tau_p(p_n) > 1- \epsilon$. Hence, we conclude that
		\begin{equation}\label{main lem-semifinite2 eq1}
			\text{for all } \mu \in K \text{ there exists } x_\mu \in \Phi_\epsilon \text{ such that } \mu(x_\mu) < \epsilon/2.
		\end{equation}
		Let $C(K)$ be the space of all scalar-valued continuous functions on $K$, where $K$ is equipped with the weak* topology induced from $M^*$, and define a linear map $h: \Phi_\epsilon \to C(K)$ by $h(x):= h_x$, where $h_x(\mu)= \mu(x)$ for all $x \in \Phi_\epsilon$ and $\mu \in K$. Furthermore, consider the following set
		\begin{align*}
			\Psi:= \{f \in C(K): f < \epsilon\}.
		\end{align*}
		
		We first claim that there exists $x_\epsilon \in \Phi_\epsilon$ such that $\mu(x_\epsilon) < \epsilon$ for all $\mu \in K$. For that it is enough to show that $\Psi \cap h(\Phi_\epsilon) \neq \emptyset$. 
		
		If the set is empty, then we can invoke Hahn-Banach separation theorem to obtain a bounded real linear functional $\Lambda$ on $C(K)$ and $a \in \R$ such that 
		\begin{align*}
			\Lambda(h_x) \geq a > \Lambda(f) \text{ for all } f \in \Psi \text{ and } x \in \Phi_\epsilon.
		\end{align*}
		Now by Riesz representation theorem we obtain a unique regular signed Borel measure $\lambda$ on $K$ such that 
		\begin{align*}
			\Lambda(f)= \int_K f d \lambda \text{ for all } f \in C(K).
		\end{align*}
		We claim that $\Lambda$ is a positive linear functional. Suppose $f \in C(K)$ such that $0< f \leq \epsilon$ but $\Lambda(f)<0$. Then for all $n \in \N$, $(-n) f \in \Psi$. Consequently, it follows that $-n \Lambda(f) < a$ for all $n \in \N$, which is a contradiction. Hence $\lambda$ can be assumed to be a probability measure.  Also since the constant function $\frac{\epsilon}{2} \in \Psi$, we get
		\begin{equation}\label{main lem-semifinite2 eq2}
			\Lambda(h_x)= \int_K \mu(x) d \lambda(\mu) > \epsilon/2 \text{ for all } x \in \Phi_\epsilon.
		\end{equation}
		Now observe that $K \subset M^*$ and consider the barycenter $\nu$ of $\lambda$ in $K$, which is defined by the integral $\nu:= \int_K \mu d \lambda(\mu)$ in the sense:
		\begin{align*}
			\Gamma(\nu)= \int_K \Gamma(\mu) d \lambda(\mu), ~ \Gamma \in (M^*)^*.
		\end{align*}
		Since $K$ is a compact, convex set, it follows from \cite[Theorem 3.27]{rudin1991} that the integral $\nu:= \int_K \mu d \lambda(\mu)$ exists and moreover, $\nu \in K$. Therefore, by eq. \ref{main lem-semifinite2 eq2} we obtain
		\begin{align*}
			\nu(x) = \Lambda(h_x) > \epsilon/2 \text{ for all } x \in \Phi_\epsilon,
		\end{align*}
		which is a contradiction to eq. \ref{main lem-semifinite2 eq1}. Hence, $\Psi \cap h(\Phi_\epsilon) \neq \emptyset$. Thus, there exists a  $ x_\epsilon \in \Phi_\epsilon $ such that 
		\begin{enumerate}
			\item $\tau_p(x_\epsilon)  \geq 1- \epsilon $ and 
			\item $\mu(x_\epsilon)    < \epsilon ~~$ for all $ \mu \in K$. 
		\end{enumerate}
		Now consider 	$q_\epsilon = \chi_{[\frac{1}{2},1)}(x_\epsilon ) $ and since $ x_\epsilon \leq p$, so we have $ q_\epsilon \in \CP(pMp)$. We further, note that 
		\begin{enumerate}
			\item $q_\epsilon  \leq 2 x_\epsilon$, which implies  $\mu(q_\epsilon)\leq 2 \mu( x_\epsilon)  < 2\epsilon$  and 
			\item $p - q_\epsilon \leq 2 (p - x_\epsilon)$, which implies $\tau_p( p-q_\epsilon) \leq 2\tau_p(p-x_\epsilon) \leq 2 \epsilon$. 
		\end{enumerate}
		
		Thus, for $  \frac{\epsilon}{2^n}$, find $ q_{\frac{\epsilon}{2^n} } =  q_n \in \CP_0(pMp) $ such that 
		\begin{enumerate}
			\item $\mu(q_n )    <   \frac{2\epsilon}{2^{n}} = \frac{\epsilon}{2^{n-1}} ~~$ for all $ \mu \in K$ and 
			\item  $\tau_p( p-q_n) \leq \frac{\epsilon}{2^{n-1}} $.  
		\end{enumerate}
		Now consider the projection  $q := \wedge_{n \geq 1} q_n$. Observe that, 
		\begin{align*}
			\tau_p(p-q) &\leq \sum_{n \geq 1} \tau_p(p-q_n)\\
			&= \sum_{k= 1}^\infty  \frac{\epsilon}{2^{n-1}}  = 2 \epsilon < 1, \text{ as }   \epsilon < \frac{1}{2}.
		\end{align*}
		This, shows that $  q \neq 0$ and further note that  	 
		$$\mu(q) \leq \frac{\epsilon}{2^{n-1}} \text{ for all } n \in \N \text{ and }\forall~ \mu \in K.	$$
		Hence, $ 	\mu(q) = 0 \text{ for all } \mu \in K$. 
	\end{proof}

	\begin{prop}
		Let $K$ be a subset of $(M^*_1)_+$ containing singular linear functionals on a von Nuemann algebra $M$. Then the follwoing are equivalent.
		\begin{enumerate}
			\item $K$ is weak* closed.
			\item For every $p \in \CP_0(M)$, there exists $q \in \CP_0(M)$ with $q \leq p$ such that $\mu(q)=0$ for all $\mu \in K$.
		\end{enumerate}
	\end{prop}
	
	\begin{proof}
		$(1) \Rightarrow (2): $ Since $K$ is weak* closed, we have $K$ is weak* compact. Hence by Lemma \ref{main lem-semifinite2} one can find a sequence of projections in $M$ satisfying the required properties.
		
		$(2) \Rightarrow (1): $ Let $\{\mu_n\}$ be a sequence in $K$ such that $\mu_n \xrightarrow{w*} \mu$ for some $\mu \in (M^*_1)_+$. Now consider $p \in \CP_0(M)$. Then by hypothesis there exists $p_m \in \CP_0(M)$ such that $p_m \leq p$ and $\mu(p_m)=0$ for all $\mu \in K$. In particular, $\mu_n(p_m)=0$ for all $n \in \N$. Therefore, $\mu(p_m)=0$, which implies $\mu$ is singular.
	\end{proof}

	\begin{lem}\label{req cpt set}
	Let $(M, G, \alpha)$ be a non-commutative dynamical system	with a f.n state $\varphi \in M_*$.  Suppose  $e \in \CP_0(M)$ such that 
		if $\nu$ is any $G$-invariant normal state  on $M$, then $s(\nu) \leq e$. Then the set
		\begin{align*}
			K_e:= \{\mu \in (M^*_1)_+ : \mu \text{ is $G$-invariant and } \mu_{\upharpoonleft M_{1-e}} \text{ is singular} \}
		\end{align*}
		is weak*-closed.
	\end{lem}
	
	\begin{proof}
		Clearly the set $K_e$ is non-empty. Let $\{\mu_m\}$ be a sequence in $K_e$ such that $\mu_m \xrightarrow{w^*} \mu$ for some $\mu \in (M^*_1)_+$. Since $\mu_m$ is $G$-invariant for all $m \in \N$, we have $\mu$ is $G$-invariant. Write $\mu= \mu_n + \mu_s$ according to \cite[Theorem 3]{Takesaki1958} and observe that $\mu_n$ is $G$-invariant (see proof of Proposition \ref{existence of inv normal state}).
		
		Now by hypothesis, $s(\mu_n)\leq e$. Therefore, $(\mu_n)_{\upharpoonleft M_{1-e}}=0$. Hence, we obtain  $\mu_{\upharpoonleft M_{1-e}}=(\mu_s)_{\upharpoonleft M_{1-e}}$, which proves the result.
	\end{proof}

	Following Lemma establishes that the support of a weakly wandering operator and the support of a $G$-invariant state are orthogonal. 
\begin{lem}\label{support of state and support of ww}
		Suppose   $x_0 \in M_+$ is any weakly wandering operator and $\nu$ be a $G$-invariant  normal state on $M$, then $s(\nu) \perp s(x_0)$.
	\end{lem}
	
	\begin{proof}
		Since $x_0 \in M_+$ is a weakly wandering operator, we have  $	\lim_{n \to \infty} \norm{A_n(x_0)} =0$. Since $\nu$ is $G$-invariant,  hence, for all $n \in \N$, we have 
		\begin{align}\label{invariant-convergence}
			\nu(x_0)=  \frac{1}{m(K_n)} \int_{K_n} \nu(\alpha_g(x_0)) dm(g) =\nu(A_n(x_0))
			\leq \norm{ A_n(x_0) }  \xrightarrow{n \rightarrow \infty} 0.
		\end{align}
		So, $\nu(x_0)= 0$, which implies $\nu(s(x_0))= 0$. Therefore,  $s(\nu) \perp  s(x_0)$.
	\end{proof}

	\begin{thm}\label{weakly-wandering}
		Let $(M, G, \alpha)$ be a non-commutative dynamical system	with a f.n state $\varphi \in M_*$ and $e \in M$ be a non-zero projection. Then the following statements are equivalent.
		\begin{enumerate}
			\item There exists a $G$-invariant normal state $\rho$ on $M$ with $s(\rho)= e$ such that,
			if $\nu$ is any $G$-invariant normal state  on $M$, then   $s(\nu) \leq e$.
			\item There is a weakly wandering operator $x_0 \in M_+$ with support $s(x_0)= 1 - e$ such that,  if $x\in M_+$ is any   weakly wandering operator, then $s(x) \leq 1-e$.
		\end{enumerate}
	\end{thm}

	\begin{proof}
		$(1) \Rightarrow (2): $ Let $p \in \CP_0(M)$ be such that $p \leq 1-e$.
		We first claim that there exists a $q \in \CP_0(M)$ such that $q \leq p$ and $\inf_{n \in \N} \varphi(A_n(q))=0$. Indeed, if this is not true, then there exists $0 \neq p \leq 1-e$ such that $\inf_{n \in \N}$ $\varphi(A_n(q))>0$ for all $0 \neq q \leq p$, which implies that there exists a $G$-invariant normal state $\nu_\varphi$ on $M$ with the support projection $s(\nu_\varphi)\geq p$ (see Proposition \ref{existence of inv normal state}). Thus, $\inf_{n \in \N} \varphi(A_n(q))=0$ for some $ q \in \CP_0(M)$ with $q \leq p$. Therefore, we obtain a subsequence $(n_k)$ such that  $\lim_{k \rar \infty } \varphi(A_{n_k}(q))=0$, as $\varphi$ f.n state so $ A_{n_k}(q) \xrightarrow{ k \rar \infty } 0 $ in SOT, equivalently $\lim_{k \rar \infty } \mu(A_{n_k}(q))=0$ for all $\mu \in M_*$.\\\\
				Now consider the set
		\begin{align*}
			K_e:= \{\theta \in (M^*_1)_+ : \theta \text{ is $G$-invariant and } \theta_{\upharpoonleft M_{1-e}} \text{ is singular} \}.
		\end{align*}
		By Lemma \ref{req cpt set}, the set $K_e$ is weak*-compact. Then by Lemma \ref{main lem-semifinite2} there exists $0 \neq q' \leq q$ such that $\theta(q')=0$ for all $\theta \in K_e$.\\
		We wish to show that $\theta(A_{n_k}(q')) \to 0$ for all positive $\theta \in M^*_1$. 
		But as $ q' \leq q$, so for the time being we have $\lim_{k \rar \infty } \mu(A_{n_k}(q'))=0$ for all $\mu \in M_*$.

		Let $\theta \in M^*_1$ be positive and consider a subsequence $(m_l)$ of $(n_k)$. Then writing $\theta= \theta_n + \theta_s$ in accordance with \cite[Theorem 3]{Takesaki1958}, we obtain
		\begin{align*}
			A^*_{m_l}(\theta) = A^*_{m_l}(\theta_n) + A^*_{m_l}(\theta_s) \text{ for all } l \in \N.
		\end{align*}
		Hence, by Banach-Alaglou theorem we can find a subsequence $(t_r)$ of $(m_l)$ such that $A^*_{t_r}(\theta_s) \xrightarrow{w^*} \psi$, for some positive $\psi \in M^*_1$. Note that by F\o lner condition, $\psi$ is $G$-invariant. Then again writing $\psi= \psi_n + \psi_s$ according to \cite[Theorem 3]{Takesaki1958}, we observe that $\psi_n$ is $G$-invariant (see proof of Proposition \ref{existence of inv normal state}). But by hypothesis, $s(\psi_n) \leq e$. Hence, $\psi_n(q')=0$. Consequently, we derive that $\theta_n(A_{t_r}(q')) \to 0$  and $\theta_s(A_{t_r}(q')) \to \psi_s(q')$.
		
		Note that $\psi_s \in K_e	$.  Therefore, $\psi_s(q')=0$ and hence $\theta(A_{t_r}(q')) \to 0$. Consequently, $A_{n_k}(q') \to 0$ weakly. Now by Theorem \ref{mean erg thm}, we have $\norm{A_n(q')} \to 0$. Hence, $q'$ is a weakly wandering projection.
		
		Let, $\{ q_j \}_{j \in \Lambda}$ be the maximal family of mutually orthogonal weakly wandering projections in $M$ such that $q_j \leq 1-e$ for all $j \in \Lambda$. Since, $M$ is $\sigma$-finite, we may take $\Lambda= \N$.
		
		We claim that, $\bar{q}:= \displaystyle \sum_{j=1}^{\infty} q_j = 1-e$. We note that, $\bar{q} \leq 1-e$. Now if $\bar{q} \neq 1-e$, then by the same construction, we get a weakly wandering  non-zero subprojection of $1-e- \bar{q}$, which is a contradiction to the maximality of the family of projections $\{ q_j \}_{j \in \N}$.
		
		Now define, $x_0:= \displaystyle \sum_{j=1}^{\infty} \frac{1}{2^j} q_j \in M_+$. Since $\bar{q}= 1-e$, we have $s(x_0) = 1-e$. We claim that $x_0$ is a weakly wandering operator. Indeed, for all $n,m \in \N$, we have 
		\begin{align*}
			\norm{A_n (x_0)}
			& \leq \sum_{j=1}^{m} \frac{1}{2^j} \norm{ A_n q_j} + \frac{1}{2^{m}} \norm{A_n(\sum_{j=1}^{\infty} \frac{1}{2^j} q_{m+j})}\\
			& \leq \sum_{j=1}^{m}  \norm{ A_n q_j} + \frac{1}{2^{m}} \norm{\sum_{j=1}^{\infty} \frac{1}{2^j} q_{m+j}} \\
			& \leq \sum_{j=1}^{m}  \norm{ A_n q_j} + \frac{1}{2^{m}} \norm{\sum_{j=1}^{\infty} q_{m+j}} \\
			& \leq \sum_{j=1}^{m}  \norm{ A_n q_j} + \frac{1}{2^{m}}.
		\end{align*}
		Let $\epsilon>0$. We choose $m \in \N$ such that $\frac{1}{2^m} < \frac{\epsilon}{2}$. Since for all $j \in \{1,\ldots, m \}$, $q_{2,j}$ is a weakly wandering, there exists $N_j \in \N$ such that for all $n \geq N_j$, $\norm{A_n q_j}< \frac{\epsilon}{2^j}$.
		
		We choose, $N:= \{ N_1, \ldots, N_m, m \} \in \N$. Hence,  $\norm{A_n (x_0)} \leq \epsilon$ for all $n \geq N$. Therefore, $x_0 \in M_+$ is a weakly wandering operator with $s(x_0)= 1-e$.
		
		$(2) \Rightarrow (1): $ Suppose there is no $G$-invariant normal state on $M$. Let $\theta \in (M^*_1)_+$ and consider the sequence $\{A_{n_k}^*(\theta)\}$ in $M^*_1$. By Banach-Alaglou theorem, there is a subsequence $(m_l)$ of $(n_k)$ such that $A^*_{m_l}(\theta) \xrightarrow{w^*} \overline{\theta}$ for some $\overline{\theta} \in (M^*_1)_+$. Since $\overline{\theta}$ is $G$-invariant, its normal component is zero (since it is $G$-invariant). Hence we have $\overline{\theta}$ is singular.
		
		Consider the set
		\begin{align*}
			K_0:= \{\theta \in (M^*_1)_+ : \theta \text{ is $G$-invariant and singular}\}
		\end{align*}
		By Lemma \ref{req cpt set}, $K_0$ is weak*-closed. Hence by Lemma \ref{main lem-semifinite2}, there is $p \in \CP_0(M)$ with $p \leq e$ such that $\mu(p)=0$ for all $\mu \in K_0$. Since $\overline{\theta} \in K_0$ for all $\theta \in (M^*_1)_+$ we have $\overline{\theta}(p)=0$. Hence, $\theta(A_{m_l}(p)) \to 0$. Therefore, $A_{n_k}(p) \to 0$ weakly, and by Theorem \ref{mean erg thm}, we have $\norm{A_n(p)} \to 0$. Hence, $p$ is a weakly wandering projection, which is a contradiction. 
		
		Finally we conclude that there is a non-zero $G$-invariant normal state on $M$. By Lemma \ref{support of state and support of ww}, it follows that if $\nu$ is any $G$-invariant normal state on $M$, then its support $s(\nu) \leq e$.
		
		Let $\mu$ be a $G$-invariant, normal state with maximal possible support. Clearly, $s(\mu) \leq e$. If $s(\mu) \neq e$, then consider the projection $1-s(\mu)$ and observe that $(1) \Rightarrow (2)$ implies the existence of a weakly wandering operator $x \in M_+$ such that $s(x)= 1 - s(\mu) \geq 1-e$, which is a contradiction.
	\end{proof}
	
	\begin{rem}	
The current approach is fundamentally different from the previous method for finding weakly wandering operators. The conventional approaches (see \cite{Bik-Dip-neveu}) for finding weakly wandering operators was limited to group actions on finite von Neumann algebras.   
  Actually, earlier techniques, implicitly find wandering projection, i.e, it finds a $q \in \CP_0(M)$ and a sequence $ \{ g_1, g_2,\cdots  \} \subseteq G$ such that $ \alpha_{g_i}(q) \perp \alpha_{g_j}(q) $ for all $ i \neq j$. Using this it was shown that $$ \lim_{n \rar \infty } \norm{A_n(q)} = 0.$$ \\
In contrast, the current techniques first show  that there exits a $ q \in \CP_0(M)$ such that 
		$$ \lim_{n \rar \infty }\mu ( A_n(q)) = 0, \text{ for all } \mu \in M^*.$$
Then  Theorem \ref{mean erg thm} is employed to conclude that $\lim_{n \rar \infty } \norm{A_n(q)} = 0$. 
Furthermore, the current idea applies  to every amenable semigroup actions on any semifinite von Neumann algebra. Additionally, as we will also see later that  it can be used for amenable group actions on any von Neumann algebra. 
	\end{rem}

	\begin{thm}[Neveu Decomposition] \label{neveu decomposition}
		Let $M$ be a semifinite von Neumann algebra and $( M, G , \alpha) $ be a non-commutative dynamical system.  Then there exist two projections $e_1, e_2 \in M$ such that $e_1 + e_2 = 1$ and 
		\begin{enumerate}
			\item there exists a $G$-invariant normal state $\rho$ on $M$  with support $s(\rho) = e_1$ and
			\item there exists a weakly wandering operator $x_0 \in M$ with support $s(x_0)= e_2$.
		\end{enumerate}
		Further, $s(\rho)$ and $s(x_0)$ are  unique. 
	\end{thm}
\begin{proof}
Let $ e_1= e $ be maximal projection of a $G$-invariant normal state on $M_*$  as in Theorem \ref{weakly-wandering} and suppose $e_2 = 1-e_1$. The rest of the proof essentially follows from Theorem \ref{weakly-wandering}.   
	
\end{proof}

\begin{rem}\label{L^1-Neveu}
	Let $ M$ be  a semifinite von Neumann algebra with f.n semifinite trace $\tau$. Then
	note that Neveu decomposition can be stated in terms of  an action of an amenable semigroup $G$ on $L^1(M, \tau)$. Let $ ( L^1(M, \tau), G, \gamma )$ be non-commutative dynamical system. 
	Then there exist two projections $e_1, e_2 \in M$ such that $e_1 + e_2 = 1$ and 
	\begin{enumerate}
		\item there exists a $Y \in L^1(M, \tau) $ such that $ \gamma_g(Y) = Y$ and  $s(Y) = e_1$. 
		\item there exists a weakly wandering operator $x_0 \in M$ with $ A_n(\gamma^*)(x_0) \rar 0$ in norm  and  support $s(x_0)= e_2$.
	\end{enumerate}
	Further, $s(Y)$ and $s(x_0)$ are  unique. \\
\end{rem}	

Now we discuss the invariance of the projections obtained in Neveu decomposition.  To be more precise, 
let $M$ be a semifinite von Neumann algebra with f.n semifinite trace $\tau$ and $ ( L^1(M, \tau), G, \gamma )$ be  a non-commutative dynamical system with  $ e_1$ and $e_2$ being   the projection obtained in Neveu decomposition (see Remark \ref{L^1-Neveu}). Then we like to discuss  whether $ s(\gamma_g(e_i)) \leq e_i$ for $i = 1, 2$ and for all $g \in G$.  
 We begin with the following definition. 
\begin{defn}
Let $ \gamma:  L^1(M, \tau) \rar L^1(M, \tau) $ be a  positive operator. Then $\gamma$ is called  Lamperti operator if $ e_1, e_2 \in L^1(M, \tau) \cap \CP_0(M)$ with $e_1 e_2 = 0$, then 
$ \gamma(e_1)   \gamma(e_2) = 0 $.
\end{defn}	
We consider the following set 
$$ \CL= \{  \gamma:  L^1(M, \tau) \rar L^1(M, \tau) \text{ positive Lamperti  contraction} \}.$$

The following result is straightforward and well known in the literature. For the completeness of this article, we include a proof. 
\begin{prop}\label{lamperti}
Let $ \gamma:  L^1(M, \tau) \rar L^1(M, \tau) $ be a  positive contraction. Then $\gamma$ is Lamperti operator if and only if  for $ a, b \in  L^1(M, \tau)_+ $ with $ ab =0$ then $ \gamma(a) \gamma(b) = 0$.	
\end{prop}	
\begin{proof}
Suppose $ a, b \in  L^1(M, \tau)_+ $ with $ ab =0$. Immediately, note that $s(a)s(b) = 0$.   Then the proof follows from the fact that there exists two sequences  $(a_n)$ and $ (b_n)$ in $M$ such that 
\begin{enumerate}
	\item $0 \leq a_n \leq a \text{ and } 0 \leq b_n \leq b$ such that  $\tau(s(a_n) ) < \infty $ and  $\tau(s(b_n) )  < \infty $  with 
	$$s(a_n) \nearrow s(a) \text{ and } s(b_n) \nearrow s(b) \text{ and }$$
	\item $\lim_{ n \rar \infty } \norm{ a_n -a }_1 = 0  \text{ and } \lim_{ n \rar \infty }  \norm{ b_n -b }_1 = 0 $.\\
\end{enumerate}
Indeed, note that $ a_n b_n = 0 $, so $ s(a_n) s(b_n) = 0$. This implies that $ \gamma(s(a_n)) \gamma(s(b_n)) = 0$. 
Hence, $\gamma(a_n) \gamma(b_n) = 0$ as $  a_n \leq \norm{ a_n} s(a_n) $ and $  b_n \leq \norm{ b_n} s(b_n) $. Then obtain that $\gamma(a) \gamma(b) =0$. 
\end{proof}

	\begin{prop}\label{inv-lamperti}
	Let $M$ be   a semifinite von Neumann algebra with a semifinite f.n trace $\tau$ and  $(M, G, \alpha )$   be a non-commutative dynamical system. Suppose $ e_1, e_2$ be the corresponding projection obtained in the Neveu decomposition, then we have the following
		$$ \alpha_g(e_2)  \leq e_2 \text{ and }   s( \alpha_g^*(e_1Xe_1)) \leq e_1  \text{ for all } X \in L^1(M, \tau), g 
		\in G.$$
Furthermore, if $ \alpha_g^* \in \CL $ for all $ g \in G$, then  
$ \alpha_g(e_1)  \leq e_1 	\text{ for all } g 
\in G.$
	\end{prop}
	
\begin{proof}	

To show  $\alpha_g(e_2)  \leq e_2 $, we recall that there exists a positive operator $x_0 \in M$ such that  $s(x_0)=e_2$ and it is  weakly wandering, i.e, $\norm{A_n(x_0)} \to 0$ as $n \to \infty$. We choose a sequence of projections $\{f_m\} \subseteq M$ such that 
\begin{align*}
f_m \nearrow s(x_0) \text{ and } \frac{1}{m} f_m \leq x_0 \text{ for all } m \in \N.
\end{align*}
Observe that $f_m$ is weakly wandering for all $m \in \N$. Hence, for all $g \in G$,  $\alpha_g(f_m)$ is also weakly wandering for all $m \in \N$. Since $e_2$ is the maximal support of the weakly wandering projection, so  we have $s(\alpha_g(f_m)) \leq e_2$ for all $m \in \N$. Therefore, $\alpha_g(f_m) \leq e_2$ for all $m \in \N$, which implies $\alpha_g(e_2) \leq e_2$, for all $g \in G$. \\
Then we note that 
\begin{align*}
	\tau( \alpha_g^*(e_1Xe_1) e_2) &= \tau( e_1Xe_1 \alpha_g(e_2) )\\
	&\leq \tau( e_1Xe_1 e_2 ) =0, \text{ as } \alpha_g(e_2) \leq e_2.
\end{align*}
Hence, $e_2 \alpha_g^*(e_1Xe_1)e_2 =0$, consequently, $s(\alpha_g^*(e_1Xe_1) ) \leq e_1$.

Now  assume that $ \alpha_g^* \in \CL $. Then we want to show that
$ \alpha_g(e_1) \leq e_1$, equivalently, $s(\alpha_g^*(e_2Xe_2)) \leq e_2$. Now suppose 
$\rho$ is a normal state with $s(\rho) = e_1$ with $ \rho\circ \alpha_g = \alpha_g$ for all $g \in G$. So, there exists $ Y \in L^1(M, \tau)$ such that $ \rho(x) = \tau(Y x) $ for all $ x \in M$ with $s(Y) = e_1$. Further, we have the following 
\begin{align*}
	\tau({\alpha_g}^*(X)y)= \tau(X\alpha_g(y)) \text{ for all }  X\in L^1(M, \tau) \text{ and } y \in M.
\end{align*}	
Thus, for all $x \in M$,  we observe that 
\begin{align*}
	\tau(Y x) = \rho(x)= \rho( \alpha_g(x)) = \tau(Y \alpha_g(x)) = \tau(\alpha_g^*(Y) x).
\end{align*} 
Hence, it follows that $ \alpha_g^*(Y) = Y $.  
As $ s(Y) = e_1 $, so $Y e_2Xe_2 =0$. Hence,  $\alpha_g^*(Y) \alpha_g^*(e_2Xe_2) =0$  for all $g \in G$. But we have $\alpha_g^*(Y) = Y$, therefore $Y \alpha_g^*(e_2Xe_2) = 0$. This implies $e_1 \alpha_g^*(e_2Xe_2)e_1  = 0$. 
\end{proof}

\begin{cor}\label{proj-inv}
Suppose $ (M, G, \alpha), e_1, e_2$ are  as in proposition \ref{lamperti} and assume that $ \alpha^*_g \in \CL$ for every $g \in G$. Then for all 
$ X \in L^1(M, \tau) $ and $g \in G$, we have $e_i \alpha_g^* (X) e_i =   \alpha_g^* ( e_i  X e_i ) 	$ for $i = 1, 2$. 
\end{cor}
\begin{proof}
First  note that we have the following equation $$ \tau( \alpha_{g}^*(X) y ) = \tau(X \alpha_{g}(y)  ) \text{ for all } X 
\in L^1(M, \tau) \text{ and } y \in M.$$
Suppose $X \in L^1(M, \tau) $ and $ g \in G$, then  using the proposition \ref{inv-lamperti}, we conclude the following 
\begin{enumerate}
	\item $ s( \alpha_g^*(e_1Xe_1)) \leq e_1 $ and $ s( \alpha_g^*(e_2Xe_2)) \leq e_2 $,
	\item $e_1 \alpha_g^*( e_1Xe_2) e_1 = e_1 \alpha_g^*( e_2Xe_1) e_1 = e_1 \alpha_g^*( e_2Xe_2) e_1 = 0 $,
	\item $e_2 \alpha_g^*( e_1Xe_2) e_2 = e_2 \alpha_g^*( e_2Xe_1) e_2 = e_2 \alpha_g^*( e_1Xe_1) e_2 = 0 $.
\end{enumerate}
Thus, we have $e_1\alpha_g^*(X) e_1 = e_1\alpha_g^*(e_1Xe_1) e_1 = \alpha_g^*(e_1Xe_1).$ Similarly, we will have $e_2\alpha_g^*(X) e_2= \alpha^*_g( e_2Xe_2)$.
\end{proof}

	\section{\textbf{Neveu decomposition for actions on any von Neumann algebra }}
	
	In this section we study Neveu Decomposition of a covariant system on arbitrary von Neumann algebra $M$ with an assumption that it commutes with the modular automorphism group associated to a f.n semifinite  weight $\psi$ on $M$.   We refer $(M, \psi)$ as non-commutative measure space.  \\\\
	Let  $L^2(M, \psi)$ be the  GNS Hilbert space associated to the f.n semifinite  weight  $\psi$ on  $M$ and  $ \pi_{\psi}, \Delta_\psi, J_\psi$ and $\sigma^\psi = ( \sigma^\psi_t)_{t \in \R }$ be the corresponding  GNS representation,  modular operator, modular conjugation operator and modular automorphism $M$  respectively. To simplify the notations, sometimes we write $\CH = L^2(M, \psi)$, identify $ \pi_{\psi}(x)$ as $x$  for all $ x \in M$ and write the  modular automorphisms group as  $\sigma = ( \sigma_t)_{t \in \R }$. 
Let $ M \rtimes_{\sigma^\psi} \R$ be the crossed product  constriction of $M$ with respect to its modular automorphisms group $\sigma^\psi$. We denote it by $N$, i.e, $ N= M \rtimes_{\sigma^\psi} \R$. Indeed, it is defined as follows. 
Consider the Hilbert space $L^2( \R, \CH) = \CH \otimes L^2(\R, m)$. For every $x \in M$ and $t \in \R$, define 
\begin{align*}
&(\pi_{\sigma^\psi} (x)\xi) (s) := \sigma_{-s}(x) \xi(s) \text{ and }\\ 
&(\lambda(t) \xi) (s) := \xi(s-t), \text{ for all } s \in \R \text{ and } \xi \in L^2( \R, \CH). 
\end{align*}	
Naturally, 	$\lambda(t) $ is expressed as $\lambda(t) = 1 \otimes \lambda_t $, where $ \lambda_t \in \CB(L^2(\R, m))$ is defined by 
$(\lambda_t(g) )(s) = g( s-t)$ for all $s \in \R$ and $g \in   \CB(L^2(\R, m))$. Then the crossed product  $M \rtimes_{\sigma^\psi} \R$
of $M$ with the action of modular automorphisms group $ \sigma^\psi$ is the von Neumann algebra generated by $\pi_{\sigma^\psi}(M)$ and $ \lambda(\R)$ i.e, 
$$ 	M \rtimes_{\sigma^\psi} \R = \{\{\pi_{\sigma^\psi} (x): x \in M \} \cup \{ \lambda(t) : t \in \R \} \}''.$$
	
	\begin{defn}
		Let $ (M, G, \alpha)$ be a non-commutative dynamical system. It is called Markov if there exists f.n semifinite weight $\psi$ on $M$ such that $ \alpha_g\circ \sigma^\psi_t = \sigma^\psi_t \circ \alpha_g $ for all $g \in G$ and $ t \in \R$. 
	\end{defn}

\begin{lem}\label{extention}
		Let $(M, \psi)$ be a non commutative measure space with a f.n semifinite  weight $\psi$ and $\sigma = ( \sigma_t)_{t \in \R }$ be its modular automorphisms group with respect to $ \psi$. Suppose $T : M \rar M$ is a positive contraction satisfying  $T \circ \sigma_t = \sigma_t \circ T  $ for all $t \in \R$. Then, there exists a unique  positive contraction $\widetilde{T} : N  \rar  N $ satisfying $$ \widetilde{T} ( \pi_\sigma( x ))  = \pi_\sigma( T(x)  \text{ and } \widetilde{T}( 1 \otimes \lambda_t) = 1 \otimes \lambda_t, \text{ for all } t \in \R .$$ 
\end{lem}
\begin{proof}
 Note that $N \subseteq M\otimes \mathcal{B}(L^2(\R))$.  Consider  $ \widetilde{T} = T \otimes 1 |_N$. Thus it is enough  to prove that  $\widetilde{T}(N) \subseteq N$, as the map $ T\otimes 1 : M\otimes \mathcal{B}(L^2(\R)) \rar M\otimes \mathcal{B}(L^2(\R)) $ is a positive map. Let $ I$ be a directed set of finite Borel partition of $G$. Suppose $ \omega_1 , \omega_2 \in I$, then we say $ \omega_1 \leq \omega_2 $ if $\omega_2$ is a refinement of $\omega_1$. Let $x \in M$ and $ \omega \in I $ with $\omega = \{ P_1, P_2, \cdots, P_n \}$ and take a set $ \{ g_1, g_2, \cdots, g_n\}$ with $ g_i \in P_i$ for $i = 1, 2, \cdots, n$,  then consider the following element 
 $$ x_\omega = \sum_{i=1}^{n } \sigma_{g_i^{-1}}(x) \otimes m_i,$$
where $m_i $ is the multiplication operator with respect to the function $ 1_{P_i}$ on $ L^2(\R)$. Thus $ \{ x_\omega : \omega \in I\}$ is a net in  $M\otimes \mathcal{B}(L^2(\R))$. Let $ \eta = \xi \otimes f $, where 
$ \xi \in \CH_{\psi} $ and $f \in L^2(\R)$. Then we have the following 
\begin{align*}
 \norm{ \pi_{\sigma} (x) \eta - x_\omega \eta }^2 &= \norm{\pi_{\sigma} (x)  (\xi \otimes f)-  x_\omega  (\xi \otimes f)}^2\\
 &= \int \norm{\pi_{\sigma} (x)  (\xi \otimes f)(s)-  x_\omega  (\xi \otimes f)(s)}^2d\lambda(s)\\
 &= \int \norm{  f(s)\sigma_{s^{-1} }(x)\xi  -  \sum_{i=1}^{n } \sigma_{g_i^{-1}}(x) \otimes m_i(\xi \otimes f)(s)}^2d\lambda(s)\\
 &= \sum_{i=1}^{n } \int_{P_i} \norm{  \sigma_{s^{-1} }(x)\xi  -  \sigma_{g_i^{-1}}(x)\xi}^2 \abs{f(s)}^2  d\lambda(s).
\end{align*}
Suppose  $ \epsilon >0$, find a compact set $K \subseteq \R$ such that $ \int_{ \R \setminus K} \abs{f(s)}^2  d\lambda(s) < \epsilon^2$.  Furthermore, choose a finite Borel partition $\{ K_1, K_2, \cdots, K_m \}$ of  $K$ such that $$ \norm{  \sigma_{s^{-1} }(x)\xi  -  \sigma_{t^{-1}}(x)\xi}  < \epsilon, \text{ for all } s, t \in K_i \text{ and }  1 \leq i \leq m.$$
Now consider the partition $ \omega = \{ K_1, K_2, \cdots, K_m,  K_{m+1} \} $ where $ K_{m+1} = G\setminus K $, then note that 
\begin{align*}
\norm{ \pi_{\sigma} (x) \eta - x_\omega \eta }^2 
 &= \sum_{i=1}^{m+1 } \int_{K_i} \norm{  \sigma_{s^{-1} }(x)\xi  -  \sigma_{g_i^{-1}}(x)\xi}^2 \abs{f(s)}^2  d\lambda(s)\\
 &<  (\norm{f}^2 + (2 \norm{x} \norm{\xi})^2) \epsilon^2. 
\end{align*}
Thus, this shows that the net $ \{ x_\omega : \omega \in I  \}$  converges to $ \pi_{\sigma} (x)$ in strong operator topology. 
Thus, we note that 
\begin{align*}
T\otimes 1 ( x_\omega ) &= T\otimes 1 (\sum_{i=1}^{n } \sigma_{g_i^{-1}}(x) \otimes m_i)\\
&= \sum_{i=1}^{n } T (\sigma_{g_i^{-1}}(x)) \otimes m_i\\
&= \sum_{i=1}^{n } \sigma_{g_i^{-1}}(T(x)) \otimes m_i, \text{ as } 	T \circ \sigma_t = \sigma_t \circ T
\end{align*}
Therefore, $\{ T\otimes 1 ( x_\omega ):  \omega \in I \} $ converges to $ \pi_{\sigma} (T(x)) $ strongly. Hence, $ (T \otimes 1) ( \pi_{\sigma} (x)  ) = \pi_{\sigma} (T(x)) $ and $ (T \otimes 1) ( 1 \otimes \lambda_t) = 1 \otimes \lambda_t $ for all $ t \in \R$. Thus,  $\tilde{T} ( N) \subseteq N$. 
\end{proof}

Let $M$ be any von Neumann algebra with f.n semifinite weight $\psi$. 
Then recall the dual weight, let  $\widetilde{\psi} $ be the dual weight on $N$. It is defines as follows 
	$$ \widetilde{\psi}( x) = \psi \circ \pi_{\sigma^\psi}^{-1} (\gamma x), ~~ x \in N_+$$
where for every $x \in N_+$, $\gamma(x)$ preparatorily  defined as 
	\begin{align*}
	\gamma(x)(\omega):= \int_\R \omega(\hat{\alpha}_p(x)) dp ~~\text{for all } \omega \in N_*^+ .
	\end{align*}
For the detailed definition we refer \cite{Hiai2020}. 
We further recall that $N$ is a semifinite von Neumann algebra with a  f.n.s trace $\tau$ defined by 
	$$ \tau( x ) = \lim_{\epsilon  \to 0} \widetilde{\psi} ( B_\epsilon x  B_\epsilon)$$
	Here $ B $ is the non-singular generator of $( \lambda(-t))$ and $ B_\epsilon = \frac{B}{1+\epsilon B}$.

\begin{lem}\label{SOT-inf}
Let $M$ be any von Neumann algebra and $( M, G,  \alpha) $ be a dynamical system. Then the following are equivalent.
\begin{enumerate}
	\item Suppose there exists a $p \in \CP_0(M)$ such that $A_{n_k}(p) \rar 0$ in SOT for some subsequence $(n_k)$.
	\item Support of  Maximal $G$-invariant state on $M$ is not $1$. 
	\item There exists a non-zero weakly wandering operator for $( M, G,  \alpha) $.
\end{enumerate}

\end{lem}
\begin{proof}
The proof follows from Theorem 	\ref{supp of maximal state}.
\end{proof}

		\begin{lem}\label{weakly-wandering-proj}
		Let $(M, \psi)$ be a non commutative measure space with a f.n.s weight $\psi$ and $( M, G,  \alpha) $ 
		be a non commutative dynamical system such that $$ \sigma_t^\psi \circ \alpha_g = \alpha_g \circ \sigma_t^\psi, \text{ for all } t \in \R \text{ and }, g \in G.$$ 
	 Suppose  that  support of maximal $G$-invariant state is not full, 
		then there exists a  $p \in \CP_0( M)$ such that $A_n(\alpha)(p) \xrightarrow{\norm{\cdot}} 0$. 
	\end{lem}

	\begin{proof}
Let $N= M\rtimes_{\sigma^\psi}\R$. Note that for each $ g \in G $, $ \alpha_g : M \rar M	$ is a unital positive map and 	$ \sigma_t^\psi \circ \alpha_g = \alpha_g \circ \sigma_t^\psi$, so, by 
Lemma \ref{extention}, there exists a unital positive map $ \widetilde{ \alpha}_g : N \rar N$ such that 
$$ \widetilde{\alpha}_g ( \pi_\sigma( x ))  = \pi_\sigma( \alpha_g(x)  \text{ and } \widetilde{\alpha}_g( 1 \otimes \lambda_t) = 1 \otimes \lambda_t, \text{ for all } t \in \R .$$ 
Write $ \widetilde{ \alpha} = ( \widetilde{ \alpha}_g )_{g \in G}$, then it is straightforward to  check that $ (N, G, \widetilde{ \alpha} )$ is a non-commutative dynamical system and further note that $N$ is a semifinite von Neumann  algebra.   		
		
Then by Theorem \ref{neveu decomposition} there exits $ e_1, e_2 \in \CP_0(N)$ such that 
	\begin{enumerate}
	\item there exists a $G$-invariant normal state $\rho$ on $N$  with support $s(\rho) = e_1$ and
	\item there exists a weakly wandering operator $x_0 \in N$ with support $s(x_0)= e_2$.
\end{enumerate}
First suppose that $e_2 =0$, and note that $\rho$ is f.n  $G$-invariant state on $N$. Now consider $\rho_0(x)=
\rho(\pi_{\sigma}(x))$ for all  $ x \in M$. Note that $\rho_0\circ \alpha_g= \rho_0$ for all $g \in G$ and $\rho_0$ is a f.n state on $M$, in other words $s(\rho_0)=1$. This is a contradiction to our assumption. Consequently, we have $e_2 \neq 0$. 
\\\\
Now assume that  $ e_2 \neq 0$. 		
Observe that 
$x_0\in N \subseteq M\otimes \mathcal{B}(L^2(\R))$. Choose a unit vector $\xi \in L^2(\R)$ such that $ (1\otimes p_\xi)  x_0 (1\otimes p_\xi) \neq 0$. Since $ x_0 \in  N \subseteq M\otimes \mathcal{B}(L^2(\R))$, so we have $ (1\otimes p_\xi) x_0 (1\otimes p_\xi) = x'_0 \otimes p_\xi$ for some non zero positive operator $x'_0 \in M$. Then, observe the following 
\begin{align*}
\norm{ A_n(\alpha)(x'_0)}&= \norm{ A_n(\alpha \otimes 1)(x'_0\otimes p_\xi)}\\
&=\norm{ A_n(\alpha \otimes 1)((1\otimes p_\xi) x_0 (1\otimes p_\xi))}\\
&=\norm{ (1\otimes p_\xi) A_n(\alpha \otimes 1)x_0 (1\otimes p_\xi)}\\
&=\norm{ (1\otimes p_\xi) A_n(\tilde{\alpha })x_0 (1\otimes p_\xi)}\\
&\leq \norm{  A_n(\tilde{\alpha })x_0 }\xrightarrow{n \rightarrow \infty} 0.\\
\end{align*}
Thus, $x'_0$ is a weakly wandering operator for $(M, G, \alpha)$. 
Now choose a non zero  projection 	$p \in M$ such that $\frac{1}{m} p \leq x_0$, for some $m \in \N$, to do this one can use spectral theorem. 
Then note that $$ \norm{ A_n(\alpha)(p)} \leq m  \norm{ A_n(\alpha)(x_0)} \xrightarrow{n \rightarrow \infty} 0.$$
Hence, $p$ is  a  non zero weakly wandering projection for $(M, G, \alpha)$. 
\end{proof}

For each $n \in \N$, consider the compact set $ F_n = [-n, n]\subseteq \R$ , then note that $(F_n)_{n \in \N}$ is a F\o lner sequence of $\R$. 	Suppose $\lambda$ is the Lebesgue measure on $\R$.
\begin{lem}\label{weakly-wandering-op}
Let $(M, G, \alpha)$ be a Markov non-commutative dynamical  system and  $p\in  \CP_0(M)$ is weakly wandering. For fix   $ n, k \in \N$, consider   $$x_{k, n}:= \frac{1}{m(F_k)} \frac{1}{m(K_n)}  \int_{K_n} \int_{F_k}  \alpha_g \circ \sigma_t(p) dm(g)d\lambda(t), $$ then $x_{k, n}$  is also weakly  wandering and $s(x_{k, n})= \bigvee_{g \in K_n}\bigvee_{t \in F_k} \alpha_g\circ \sigma_t(p)$.
\end{lem}

\begin{proof}
Consider $x_n = 	\frac{1}{m(K_n)}  \int_{K_n} \alpha_g (p) dm(g)$. Then note that 
\begin{align*}
x_{k, n} &=  \frac{1}{m(F_k)} \frac{1}{m(K_n)}  \int_{K_n} \int_{F_k}  \alpha_g \circ \sigma_t(p) dm(g)\\
&= \frac{1}{m(F_k)}  \int_{F_k}   \sigma_t \Big(   \frac{1}{m(K_n)}  \int_{K_n} \alpha_g (p) dm(g) \Big) d\lambda(t), \text{ as } \sigma_t \circ \alpha_g = \alpha_g \circ \sigma_t\\ 
&= \frac{1}{m(F_k)}  \int_{F_k}   \sigma_t (  x_n ) d\lambda(t).
\end{align*}
Thus, observe that whenever $x_n$ is weakly  wandering, then $x_{k, n} $ is also weakly  wandering. 
So it is enough to prove that 
$x_n$ is weakly  wandering. Indeed, for all $l \in \N$, note that 
\begin{align*}
	A_l(x_n) &= \frac{1}{m(K_l)} \int_{K_l} \alpha_g(x_n) dm(g) \\
	&= \frac{1}{m(K_l)} \frac{1}{m(K_n)}  \int_{K_l}  \int_{K_n} \alpha_{gh}(p) dh dg \\
	&= \frac{1}{m(K_n)} \frac{1}{m(K_l)}  \int_{K_n}  \int_{K_l} \alpha_{gh}(p) dg dh \\
	&= \frac{1}{m(K_n)} \frac{1}{m(K_l)}  \int_{K_n} \Big[ \int_{K_l} \alpha_{gh}(p) dg - \int_{K_l} \alpha_{g}(p) dg \Big] dh + \frac{1}{m(K_l)} \int_{K_l} \alpha_{g}(p) dg.
	\end{align*}
	Hence, for all $n,l \in \N$
	\begin{align*}
	\norm{A_l(x_n)} & \leq \frac{m(K_l h \Delta K_l )}{m(K_l)} \sup_{g \in K_l} \norm{\alpha_g(p)} + \norm{A_l(p)}\\
	& \leq \frac{m(K_l h \Delta K_l )}{m(K_l)} + \norm{A_m(p)}
	\end{align*}
Therefore, $\lim_{l \to \infty} \norm{A_l(x_n)} =0$.\\\\
For the other part let us denote $q= \bigvee_{g \in K_n} \alpha_g(p)$. Note that 
	\begin{align*}
	q x_n = \frac{1}{m(K_n)} \int_{K_n} q \alpha_g(p) dm(g) = \frac{1}{m(K_n)} \int_{K_n} \alpha_g(p) dm(g)= x_n.
	\end{align*}
	Therefore, $q\geq s(x_n)$. Conversely if $\xi \in q^\perp$, then
	\begin{align*}
	\inner{q \xi}{\xi}= 0 &\Rightarrow \inner{\alpha_g(p) \xi}{\xi}=0 ~~\text{ for all } g \in K_n \\
	& \Rightarrow \inner{x_n \xi}{\xi}= 0 \\
	& \Rightarrow x_n^{1/2} \xi =0 \\
	& \Rightarrow \xi \in s(x_n)^\perp.
	\end{align*}
	Therefore, $q=s(x_n)$. By the same argument, we obtain that  $$s(x_{k, n})= \bigvee_{g \in K_n}\bigvee_{t \in F_k} \alpha_g\circ \sigma_t(p)$$.
\end{proof}

	Let $\varphi$ be a  f.n state on $M$ and $e$ be the maximal projection in $M$ satisfying the following condition
\begin{align*}
\inf_{n \in \N} A_n(\varphi)(p)>0 \ \forall \ p \in \mathcal{P}(M) \text{ such that } 0 \neq p \leq e.
\end{align*}

\begin{lem}
	Let $(M, G, \alpha)$ be a Markov  covariant system. 
	Then there  exists a weakly wandering operator $x_0 \in M$ such that $s(x_0) = 1-e$. 
\end{lem}

\begin{proof}
	Since $\alpha_g(e)=e$ for all $g \in G$. Without any loss of generality we may assume that $e=0$. Then we show that there exists a weakly wandering $x_0 \in M$ such that $s(x_0) = 1$.
	
	First note that there exists a $q \in \CP_0(M)$ such that $\inf_{n \in \N} A_n(\varphi)(q)=0$. Hence, by  Lemma \ref{SOT-inf} and Lemma \ref{weakly-wandering-proj},  there exists a $p \in \CP_0(M)$  such that $p$ is weakly wandering.
	
	Now for all $n \in \N$ consider the operator
$$x_{k, n}:= \frac{1}{m(F_k)} \frac{1}{m(K_n)}  \int_{K_n} \int_{F_k}  \alpha_g \circ \sigma_t(p) dm(g)d\lambda(t), $$ 
By Lemma \ref{weakly-wandering-op},  it follows that $x_{k, n}$ is a weakly wandering operator for all $n \in \N$ with support
	$s(x_{k, n})= \bigvee_{g \in K_n}\bigvee_{t \in F_k} \alpha_g\circ \sigma_t(p).$
 Then take the following operator
	\begin{align*}
	y  = \sum_{n, k=1}^\infty \frac{x_{k, n}}{2^{n+k}}.
	\end{align*} 
	It follows that $y$ is wandering operator with support $s(y)= \bigvee_{n, k\in \N} s(x_{k, n})$. 
Actually we have the following 
\begin{align*}
s(y)&= \bigvee_{n, k\in \N} s(x_{k, n})\\
&= \bigvee_{n, k\in \N} \bigvee_{g \in K_n}\bigvee_{t \in F_k} \alpha_g\circ \sigma_t(p)\\
&= \bigvee_{g \in G}\bigvee_{t \in \R} \alpha_g\circ \sigma_t(p), \text{ since } \cup_{n \in \N} K_n =G \text{ and } \cup_{k \in \N} F_n =\R.
\end{align*}	
Now write $z_1 = s(y)$ and thus, we note that $ \alpha_g(z_1) = z_1 $ for all $g \in G$ and $ \sigma_t(z_1) = z_1 $ for all $t \in \R$. 	
Consider the von Neumann algebra $M_{1-z} =(1-z_1) M(1-z_1) $ and let $\alpha_{1,g}$ be the restriction of $\alpha_g$ on $M_{1-z} $ for all $g \in G$. Further, observe that $ (M_{1-z}, G, \alpha_{1, g} )$  is a Markov covariant system.

Then by repeating the same argument as before we obtain a $G$-invariant projection $z_2 \in \mathcal{P}_0((1-z_1) M(1-z_1))$ and a weakly wandering operator $y_2 $ with $s(y_2)= z_2$. Repeating the similar process we obtain sequence of orthonormal projections $\{ z_1, z_2, \cdots \}$ and a sequence of weakly wandering operator $\{ y_1, y_2, \cdots \}$ such that $ s(y_k) = z_k$  and $\norm{y_k} \leq 1$ for all $k \in \N$.
	Therefore, we will have $  \bigvee_{k\in \N } z_k = 1$ and consider
	\begin{align*}
	y_0 = \sum_{k=1}^\infty \frac{ y_k}{2^k}.
	\end{align*} 
	Note that $y_0$ is a weakly wandering operator with $s(y_0) = 1$.
\end{proof}
Thus, we summarize the Neveu  decomposition  for arbitrary von Neumann algebra as follow.
	\begin{thm}[Neveu Decomposition] \label{neveu decomposition-1}
		Let $M$ be any  von Neumann algebra and $( M, G , \alpha) $ be a Markov covariant system. Then there exist two projections $e_1, e_2 \in M$ such that $e_1 + e_2 = 1$ and 
		\begin{enumerate}
			\item there exists a $G$-invariant normal state $\rho$ on $M$  with support $s(\rho) = e_1$ and
			\item there exists a weakly wandering operator $x_0 \in M$ with support $s(x_0)= e_2$.
		\end{enumerate}
		Further, $s(\rho)$ and $s(x_0)$ are  unique. 
	\end{thm}

\begin{rem}
In this section, we have obtained the Neveu decomposition for Markov covariant system  $(M, G, \alpha)$ associated to the actions of a  amenable group $G$ on a  von Neumann algebra $M$. We note that if there exists f.n semifinite weight $\psi$ on $M$ such that $ \psi \circ \alpha_g = \psi $, then it is straightforward to check that $ \sigma^\psi_t \circ \alpha_g =  \alpha_g \circ \sigma^\psi_t $, for all $g \in G$ and $t \in \R$.  Hence, it is a Markov covariant system. We like to highlight that it  is natural to assume  $(M, G, \alpha)$ preserving  a f.n semifinite weight $\psi$ on $M$. Possibly, most   covariant system $(M, G, \alpha)$ preserve  a f.n semifinite weight on $M$. But, right now neither we have a proof nor a  reference. Further, we note that automorphisms on $\CB(\CH)$  always preserve the f.n semifinite trace on $\CB(\CH)$.  Furthermore, note that the shift automorphism $ \alpha: \CB(L^2(\R, \lambda)) \rar \CB(L^2(\R, \lambda ))$ where $\lambda$ is Lebesgue measure on $ \R$,   does not preserve any state and but preserve the trace on $\CB(L^2( \R, \lambda))$. 

Even in the classical setting,  given a $\sigma$-finite measure space $( X, \mu)$ and a transformation $T$ defined on it, the problem of finding a finite measure $\nu$ invariant under T is studied extensively  by many authors. We refer to \cite{Hajian1964}, \cite{ito1964invariant}, \cite{Hajian1968/1969}, \cite{hajian1972invariant} and references therein.  Thus, Markov condition on $(M, G, \alpha)$ is natural and very mild compare to the  state preserving condition.

\end{rem}

	\section{Example}
In this section we discuss some examples.  Here we assume $G$ to be an amenable semigroup or a group and $ (K_n)$ be a F\o lner sequence.
Let $ ( \CH, U, G)$ be a non-commutative dynamical system,  i.e, 
$(U_g)$ is a strongly continuous unitary presentation of $G$ on the Hilbert space $\CH$.  We begin with following definition.
\begin{defn}
	Let $G$ be a amenable semigroup and $ ( \CH, U, G)$ be a non-commutative dynamical system. It is called weakly mixing if 	
	for all $\xi,\eta \in \CH$ one  has 
	\begin{align*}
	\frac{1}{m(K_n)} \int_{K_n} \abs{\langle U_g \xi, \eta \rangle }^2 dm(g)  \rightarrow 0 \text{ as } n \rightarrow \infty.
	\end{align*}		
			
\end{defn}

	Let $\alpha$ be a automorphism on $\CB(\CH)$. Then there exists a unitary $ U \in \CB(\CH)$ such that $ \alpha(x) = Ux U^*$ for all $ x \in \CB(\CH)$. Decompose $\CH = \CH_c \oplus \CH_{wm}$, where $\CH_c $ and $ \CH_{wm}$ are compact (equivalently,  it the eigenspace of $U$) and weakly mixing part of $U$ respectively. Let $P_c$ and $P_{wm}$ are two projection onto the subspaces 
	$\CH_c $ and $ \CH_{wm}$ respectively. 
\begin{thm}
Let $ \alpha $ be an automorphism on $ \CB(\CH)$. Then there exists an $\alpha$-invariant  normal state $\varphi_c$ on $ \CB(\CH)$ such that $s(\varphi_c) = P_c$ and  a weakly wandering operator $x_{wm} \in \CB(\CH)$ such that $s(x_{wm}) = P_{wm}$. 
\end{thm}
\begin{proof}
		Let $\{ e_i\}_{i \in \N}$ be a orthonormal basis of $\CH_c $ consists of eigenvectors of $U$ and $\{ \epsilon_i\}_{i \in I}$ be a set of positive numbers such that $\sum_{i \in I} \epsilon_i = 1$. Then consider $$ \varphi_c(x) =  \sum_{i \in I} \epsilon_i \langle x e_i, e_i \rangle, \text{ for all } x \in \CB(\CH)$$
		Clearly, we have  $ \varphi_c\circ \alpha = \varphi_c$.

		Now we note that  $ U|_{ \CH_{wm}} $ is weakly mixing. Then by \cite[Theorem 1.1]{krengel1972weakly} weakly wandering vectors of  $ U|_{ \CH_{wm}} $  are dense in  $  \CH_{wm} $. Let $\{ \eta_k : k \in \N \} \subseteq \CH_{wm} $ be a dense collection of  weakly wandering vectors of   $ U|_{ \CH_{wm}}$.  Let $ \eta \in \{ \eta_k : k \in \N \}$, then there exists a subsequence  $ n_1 < n_2 < \cdots $ such that $ U^{n_i}\eta \perp U^{n_j} \eta $ for $ i \neq j$. Now we show that $ \alpha^{n_i}(P_\eta ) \perp  \alpha^{n_j}(P_\eta )$. Indeed observe that $ \alpha^{n_i}(P_\eta )  = P_{U^{n_i}\eta} $ and $ \alpha^{n_j}(P_\eta )  = P_{U^{n_j}\eta} $. Hence, $ \alpha^{n_i}(P_\eta ) \perp  \alpha^{n_j}(P_\eta )$ for all $ i \neq j$.  Now it is standard to show that $P_\eta $ is weakly wandering projection. 
		 Indeed, let $ A_n(P_\eta) = \frac{1}{n}\sum_{ k \in K_n} \alpha^k(P_\eta)$ where $ K_n = \{ 0, 1, 2, \cdots, n-1 \}$, then for $n, k \in \N$, we have
		\begin{align*}
		\norm{A_n ((P_\eta))}
		&= \frac{1}{k} \norm{k \cdot A_n (P_\eta)}\\
		&\leq \frac{1}{k} \sum_{j=1}^{k} \norm{A_n (P_\eta) - A_n(\alpha^{n_j}(P_\eta))} + \frac{1}{k} \norm{ \sum_{j=1}^{k} A_n(\alpha^{n_j}(P_\eta))}\\
		& \leq \frac{1}{k} \sum_{j=1}^{k}\frac{ \abs{(K_n \Delta K_n n_j)}}{n} + \frac{1}{k} \norm{ A_n ( \sum_{j=1}^{k} \alpha^{n_j}(P_\eta))} \\
			& \leq \frac{1}{k} \sum_{j=1}^{k}\frac{ \abs{(K_n \Delta K_n n_j)}}{n} + \frac{1}{k} \norm{ \sum_{j=1}^{k} P_{U^{n_j}(\eta)} }\\
				& = \frac{1}{k} \sum_{j=1}^{k}\frac{ \abs{(K_n \Delta K_n n_j)}}{n} + \frac{1}{k} \norm{  P_{ \sum_{j=1}^{k}U^{n_j}(\eta)} }, \text{as }  U^{n_i}\eta \perp U^{n_j} \eta, ~ i \neq j.\\
				& \leq \frac{1}{k} \sum_{j=1}^{k}\frac{ \abs{(K_n \Delta K_n n_j)}}{n} + \frac{1}{k} 
		\end{align*}
		Now let $\epsilon >0$. We choose $k \in \N$ such that $\frac{1}{k}< \frac{\epsilon}{2}$. Since $ ( K_n)$ is a F\o lner sequence, so for each  $j \in \{ 1, \ldots, k \}$, there exists a  $N_j \in \N$ such that 
		\begin{align*}
		\frac{\abs{(K_n \Delta K_n n_j)}}{n} < \frac{\epsilon}{2}, \text{ for all } n \geq N_j.
		\end{align*}
		Choose $N:= \max \{ N_1, \ldots, N_k, k \} \in \N$ and for all $n \geq N$, note that 
		\begin{align*}
		\norm{A_n(P_\eta)} \leq \epsilon.
		\end{align*}
Now we set $ x_{wm} = \sum_{k=1}^{\infty}\frac{1}{2^{k+1}} p_{\eta_k}.$ It is straightforward to note that $ x_{wm} $ is weakly wandering operator with $s(x_{wm}) = P_{wm}$. 
	\end{proof}
	
	\begin{rem}
		We note that if $ \eta \in \CH_{wm}$, then it is straightforward to check that $ \text{inf}_n A_n(P_\eta)(\varphi) = 0 $,  for  any faithful normal state in $\CB(\CH)_*$ and since $P_\eta$  is finite rank projection, so $ \mu( A_n(P_\eta)) = 0$ for all singular linear functional $\mu$ in $\CB(\CH)^*$.  Then it follows that   $ P_\eta$ is weakly wandering operator.  
	\end{rem}

For the next result assume that $G$ is an  amenable group or a semigroup with  a F\o lner sequence $\{ K_n: n \in \N\}$.

\begin{defn}

	Let $ ( \CH, U, G)$ be a  non-commutative dynamical system.  A vector $\xi \in \CH$ is called wandering if following holds 
\begin{align*}
\frac{1}{m(K_n)} \int_{K_n} U_g \xi ~ dm(g)  \rightarrow 0 \text{ as } n\rightarrow \infty \text{ in } \norm{ \cdot }.
\end{align*}

\end{defn}

Following is a weak version of Krengel Theorem. 
\begin{thm}
Let $ ( \CH, U, G)$ be a weakly mixing  non-commutative dynamical system. Then every elements of $\CH$ are weakly wandering for 	$ ( \CH, U, G)$.  
\end{thm}

\begin{proof}
Consider $ \alpha_g (x) = U_g x U_g^*$ for all $ g \in G$. Let $ x \in \CB(\CH)$,  then we note the following 
\begin{align*}
\alpha_g \circ \alpha_h (x) &= \alpha_g (\alpha_h (x) )= \alpha_g ( U_h xU_h^*)\\ &= U_g U_h x U_h^* U_g^*  = U_{gh} x U_{gh}^* 
\end{align*}
Thus, it follows that $(\CB(\CH), \alpha, G)$ is a non-commutative dynamical system. 
Suppose $\varphi$ be a f.n state on $\CB(\CH)$. Then we wish to show that for every non-zero projection $ p \in \CB(\CH)$ with $\tau(p) < \infty$, we have $\inf_{n \in \N} \varphi(A_n(p))=0$, where $\tau$ is semifinite trace on $\CB(\CH)$.  Indeed, 
 we note that $ \varphi$ can be written as $ \varphi(\cdot) = \sum_{i=1}^{\infty} \delta_i \langle (\cdot) \eta_i, \eta_i \rangle $ where $ \sum_i \delta_i = 1 $ and $\eta_i \in \CH $ with $\norm{\eta_i} = 1$.	
Now assume that $p$ is a projection onto $\C\xi$ for some $\xi \in \CH$. Then for all $n \in \N$, observe that 
\begin{align*}
\varphi(A_n(p))&= \sum_{i=1}^{\infty} \delta_i \langle (A_n(p)) \eta_i, \eta_i \rangle\\ 
&= \sum_{i=1}^{\infty} \delta_i \frac{1}{m(K_n)} \int_{K_n}  \langle \alpha_g(p)\eta_i, \eta_i \rangle dm(g)\\ 
&= \sum_{i=1}^{\infty} \delta_i \frac{1}{m(K_n)} \int_{K_n}  \langle (U_g p U_g^* )\eta_i, \eta_i \rangle dm(g)\\ 
&= \sum_{i=1}^{\infty} \delta_i \frac{1}{m(K_n)} \int_{K_n}  \langle  \langle  U_g^*\eta_i, \xi\rangle  U_g\xi,  \eta_i \rangle dm(g)\\ 
&= \sum_{i=1}^{\infty} \delta_i \frac{1}{m(K_n)} \int_{K_n}  \abs{\langle  U_g\xi,  \eta_i \rangle}^2 dm(g)\\ 
\end{align*}
As $\sum \delta_i< \infty$, it follows that $\lim_{n \rightarrow \infty } \varphi(A_n(p)) = 0 $. Then for any finite rank projection the result follows immediately. We note that if $p$ is any finite rank projection then 
$ \varphi(p)= 0 $ for all singular linear functional on $\CB(\CH)$. Then it follows from the proof of  $(1) \implies (2)$ of  Theorem \ref{weakly-wandering} that $p$ is weakly wandering for  $(\CB(\CH), \alpha, G)$. Let $\xi \in \CH$ and suppose $ p_\xi$ is projection onto $\C\xi$. Then  for all $\eta \in \CH$,  note that 
\begin{align*}
\langle  A_n(p_\xi)\eta, \eta \rangle   &=  \frac{1}{m(K_n)} \int_{K_n}  \langle (U_g p_\xi U_g^* )\eta, \eta \rangle dm(g)\\ 
&= \frac{1}{m(K_n)} \int_{K_n}  \langle  \langle  U_g^*\eta, \xi\rangle  U_g\xi,  \eta \rangle dm(g)\\ 
&= \frac{1}{m(K_n)} \int_{K_n}  \abs{\langle  U_g\xi,  \eta \rangle}^2 dm(g).\\ 
\end{align*} 
Thus, follows that  if $ \lim_{n \to \infty} \norm{ A_n(p_\xi) } = 0 $, then it implies 
 \begin{align*}
\frac{1}{m(K_n)} \int_{K_n} U_g \xi ~ dm(g)  \rightarrow 0 \text{ as } n\rightarrow \infty \text{ in } \norm{ \cdot }.
\end{align*} 
Hence, every vector of $\CH$   is weakly wandering for $ (\CH, U, G)$.  	
\end{proof}

	\section{Pointwise Ergodic theorem}
	
	In this section we study the pointwise ergodic theorem for non-commutative dynamical system $(L^1(M, \tau), G, \gamma)$. The main results of this section follows from \cite{Bik-Dip-neveu} and \cite{bikram2023noncommutative}. 
	To prove  the mean ergodic theorem, first it was required to prove a version of mean ergodic theorem in the GNS space lavel. For that   in addition to $\alpha$ being positive maps, it was assumed that for each $g\in G$, the map  $\gamma_g^*: M \rar M$  satisfies the Schwarz condition, i.e, $ \gamma_g^*(x)^* \gamma_g^*(x) \leq \gamma_g^*(x^*x)$ for all $ x\in M$. In this section, we notice that the Schwarz map condition is redundant, hence improving the results for the actions by positive maps. We begin with recalling the following known results.
	
	\begin{prop}\cite[Paulsen Prop 3.3, ex 3.4, pp-40]{Paulsen:2002tj}\label{cp imply schwarz}
		Let $M, N$ be two unital $C^*$-algebras, and $\phi: M \to N$ be a $2$-positive map. Then $\phi(a)^* \phi(a) \leq \norm{\phi(1)}\phi(a^*a)$ for all $a \in M$.
	\end{prop}
	
	\begin{prop}\cite[Theorem 3.11]{Paulsen:2002tj}\label{pstv imply cp}
		Let $M$ be the commutative $C^*$ algebra $C(L)$ and $N$ be another $C^*$ algebra. Then if $\phi: C(L) \to N$ be a positive map, then $\phi$ is completely positive.
	\end{prop} 
	
 We note that given  a non-commutative dynamical system $ ( M_*, G, \gamma)$, we have dual non-commutative dynamical system $( M, G, \alpha)$, where  $\alpha = ( \alpha_g)_{g \in G}$ and $ \alpha_g  = \gamma_g^*$ for all $ g \in G$. It is to note that $( M, G, \alpha)$ becomes a non-commutative  dynamical system. In the sequel, we use both in the interest of presentation.

		\begin{rem}\label{pstv maps are kernel}
		Let $(M, G, \alpha)$ be a non-commutative dynamical system. Then for all $g \in G$,  $\alpha_g$ is a positive contraction for all $g \in G$. Therefore, $\norm{\alpha_g(1)} \leq 1$ for all $g \in G$. Let $x \in M_s$ and consider the abelian von Neumann algebra generated by $x$, denoted by VN($x$). Then by Proposition \ref{pstv imply cp} it follows that $(\alpha_g)_{\upharpoonleft_\text{VN($x$)}}$ is completely positive, in particular $2$-positive. Hence by Proposition \ref{cp imply schwarz}, we have $\alpha_g(x)^* \alpha_g(x) \leq \alpha_g(x^*x)$.
	\end{rem}
	
	Suppose $G$ is an amenable semigroup.  If $G$ is an amenable group then by $\alpha$ we denote an action of $G$ on the von Neumann algebra $M$. On the other hand if $G$ is a semigroup then $\alpha$ will denote either an action or an anti-action of $G$ on $M$. Similarly, if $G$ is an amenable group we also consider a sequence of F\o lner sets $\{K_n \}_{n \in \N}$ as defined in \cite{Bik-Dip-neveu} and if $G$ is a semigroup we consider a net of F\o lner sets $\{K_l \}_{l \in \R_+}$ as denied in \cite{bikram2023noncommutative}. The tripple $(M, G, \alpha)$ will form a non-commutative dynamical system.

	\noindent Now we have the following mean ergodic convergence theorem.
	
	\begin{thm}\label{mean ergodic thm}
		Let $(M, G, \alpha)$ be a non-commutative dynamical system.  Also assume that there exists a  f.n state $\rho$.
		Then for all $\mu \in M_*$, there exists a $\bar{\mu} \in M_*$ such that  
		\begin{align*}
		\bar{\mu}= \norm{\cdot}_1- \lim_{n \to \infty} B_n(\mu),
		\end{align*}
		where, 
		\begin{align*}
		B_n(\cdot):= 	
		\begin{cases}
		\frac{1}{m(K_n)} \int_{K_n} \alpha^*_{g^{-1}} (\cdot) dm(g) &\text{ when  } G \text{ is amenable group},~ n \in \N, \\\\
		\frac{1}{m(K_l)} \int_{K_l} \alpha^*_g(\cdot) dm(g) &  \text{  when } G \text{ is semigroup }, ~l \in \R_+.
		\end{cases}
		\end{align*}
		Further, if  $G$ is a unimodular group and $\{ K_n\}$ are symmetric, then $\bar{\mu}$ is $G$-invariant.  
	\end{thm}
	
	\begin{proof}
		Let $L^2(M_s, \rho)$ be the closure of $M_s$ with respect to the norm induced from the inner product $ \langle \cdot, \cdot \rangle_\rho $. Define the following maps on the Hilbert space $L^2(M_s, \rho)$.
		\begin{align*}
		u_g(x \Omega_{\rho})= \alpha_g(x) \Omega_{\rho}, x \in M_s, g \in G.
		\end{align*}
		Now applying Remark \ref{pstv maps are kernel} observe that $u_g$ defines a contraction on $L^2(M_s, \rho)$. Now the rest of the proof follows verbatim from \cite[Lemma 4.5]{Bik-Dip-neveu} for the case of amenable groups and \cite[Theorem 3.4]{bikram2023noncommutative} for the case of semigroups.
	\end{proof}
	
	Before we move onto the main theorem of this section (that is the pointwise convergence theorem) we further need to fix a few notations and need recall the following definition. 
\begin{defn}\label{gr of poly growth}
	A locally compact group $G$ is said to be of polynomial growth if there exists a compact generating subset $V$ of $G$ (i.e, $\cup_{n \in \N} V^n = G$) satisfying the following condition.
	\begin{align*}
	\text{There exists } k>0 \text{ and } r \in \N \text{ such that } m(V^n) \leq kn^r \text{ for all } n \in \N.
	\end{align*}
\end{defn}
	
\begin{rem}
	It is known from \cite{Tessera2007} that if $G$ is a group as in definition \ref{gr of poly growth} then $G$ is amenable and for the compact generating set $V$, the sequence $\{ V^n \}_{n \in \N}$ satisfies the F\o lner Condition. It is also known from \cite{Guivarch1973} that a locally compact group with polynomial growth is unimodular.
\end{rem}
We like to discuss  pointwise ergodic theorem  	for the actions of polynomial growth group and for the actions of  semigroups $\Z^d_+$ and $\R^d_+$ 
 for a natural number  $d \geq 1$. For the simplicity of  notation, let $\P$ be the collection  of all LCSH polynomial growth group and write $ \G =  \P \cup \{ \Z^d_+ : d \in \N \} \cup \{ \R^d_+ : d \in \N  \}$. 
 Further,  we note that if $ G \in \P$, then we take ergodic averages with respect to the F\o lner sequences $\{ V^n \}_{n \in \N}$. 
 We also point out that the non-commutative dynamical system $( M, \Z^d_+, \alpha )$ is determined by  $d$-commuting positive contractions  $ \alpha_1, \alpha_2, \cdots ,\alpha_d$ on $M$ such that $ \alpha_{( i_1, \cdots, i_d)} (\cdot) = \alpha_1^{i_1} \alpha_2^{i_2}\cdots \alpha_d^{i_d} (\cdot)$ for $ ( i_1, \cdots, i_d) \in \Z_+^d$. For the dynamical system 
 $(M, \R^d_+, \alpha)$, we consider the ergodic avarages with respect to the set  $Q_a:= \{(t_1, \ldots, t_d) \in \R^d_+ : t_1<a, \ldots, t_d< a\}$  for $a \in \R_+$. Thus, with the preceding notations, for a dynamical system $ ( M, G, \alpha)$ where $ G \in \G$, we discuss the pointwise ergodic theorem and stochastic ergodic theorem for the following ergodic averages;

		\begin{align*}
	A_a(\cdot):= 	
	\begin{cases}
	\frac{1}{V^a} \int_{V^a} \alpha_g(\cdot) dm(g) & \text{ when } G \in \P,~ a \in \N, \\\\
	\frac{1}{a^d} \sum_{0 \leq i_1 <a} \cdots  \sum_{0 \leq i_d <a} \alpha_1^{i_1} \cdots \alpha_d^{i_d} (\cdot)&\text{ when  } G = \Z_+^d,~ a \in \N, \\\\
	\frac{1}{a^d} \int_{Q_a} \alpha_t(\cdot) dt &  \text{  when } G = \R_+^d, ~a \in \R_+.
	\end{cases}
	\end{align*}	
For $ X \in M_* $, by $ A_a(X) $ we mean $ A_a^*(X)$.  By abusing of  notation  we make no difference between $ A_a(\cdot) $ and  $ A_a^*(\cdot)$, unless it is not clear from the context. 	
 Now we have the main theorem of this section.

	\begin{thm}\label{bau conv theorem}
Let $M$ be a finite von Neumann algebra with a f.n trace $\tau$ and $(M, G, \alpha)$ be a non-commutative dynamical system with $G \in \G$.  Furthermore, assume that there exists  a f.n $G$-invariant  state $\rho$ on $M$. Then for all $Y \in L^1(M, \tau)$, there exists $\overline{Y} \in L^1(M, \tau)$ such that $M_a(Y)$ converges to $\overline{Y}$ bilaterally almost uniformly. 	
	\end{thm}
	
	\begin{proof}
		Note that in view of Theorem \ref{mean ergodic thm}, $\overline{Y}$ exists for any $Y \in L^1(M, \tau)$. Now the bilateral uniform convergence of $A_a(Y)$ to $\overline{Y}$ follows from \cite[Theorem 4.15]{Bik-Dip-neveu} and Theorem 5.5 of \cite{bikram2023noncommutative}.
	\end{proof}

	\section{Stochastic Ergodic Theorem}
	
	In this section we combine the results obtained in \S 3 and \S 6 to prove a stochastic ergodic theorem. 
	Throughout this section we assume that $M \subseteq \CB(\CH)$ is a von Neumann algebra with a f.n tracial state  $\tau$. We further assume that $G$ is a group of polynomial growth with a compact, symmetric generating set $V$ and in that 
	case the averages will be considered   with respect to the F\o lner sequence $\{ V^n \}_{n \in \N}$ or $G \in \{\Z_+^n, \R_+^n \}$. 
	Then we prove stochastic ergodic theorem for non-commutative dynamical system $ ( L^1(M, \tau), G, \gamma)$. 
	Recall that a  non-commutative dynamical system $ ( L^1(M, \tau), G, \gamma)$ consists of a map  $ G \ni g \rar \gamma_g \in \CB ( L^1(M, \tau) )$, satisfies the following 
 
\begin{enumerate}
 \item   $ \gamma_g \circ \gamma_h = \gamma_{gh}$,   for all $g, h \in G$ and  for all $ x \in L^1(M, \tau)	$ the map $ g \rar \gamma_g(x)$ is continuous.
 \item For all $g \in G$, the map $ \gamma_g : L^1(M, \tau) \rar  L^1(M, \tau)$ is  a  positive contraction, i.e,  $ \gamma_g(x ) \geq 0$ and $  \norm{ \gamma_g(x) } \leq \norm{x}$ for all $x \in L^1(M, \tau)_+$.

\end{enumerate} 
Given a  non-commutative dynamical system $( L^1(M, \tau), G, T)$, consider dual map  $  \gamma_g^* : M \rar M$  and note that $\gamma_g^* : M \rar M$ is a subunital positive contraction and write $\gamma^* = (\gamma_g^*)$, then 
 $(M, G, \gamma^* )$  becomes  a non-commutative dynamical system. Let $ e_1$ and $e_2$ be the two  projections obtained in Theorem \ref{neveu decomposition} which satisfy the  following.
		\begin{enumerate}
			\item There exists a normal state $ \rho $ on $M$  with $s(\rho) = e_1$ such that $ \rho(\gamma_g^*(x)) = \rho(x)$ for all $g \in G$ and $ x \in M$,  equivalently there exits a $Y \in L^1(M, \tau) $ with $ s(Y) = e_1$ and   $\gamma_g(Y) = Y$ for all $ g \in G$. 	\\
			\item $ \gamma_g ( e_1 L^1(M, \tau) e_1) \subseteq  e_1 L^1(M, \tau) e_1 $ for all $g \in G$.\\
			
			\item $ \gamma_g^*(e_2) \leq e_2$  for all $g \in G$. \\	
			
\end{enumerate}	
We write $M_{e_i}= e_i Me_i$ and  $\tau_{e_i}= \frac{1}{\tau(e_i)}\tau|_{e_iMe_i}$ for $i = 1, 2$.	
Then note that $ 		e_1 L^1(M, \tau) e_1 = L^1(M_{e_1}, \tau_{e_1}) $. Further, for all $g \in G$,  consider 
$ \gamma_g|_{L^1(M_{e_1}, \tau_{e_1})} : L^1(M_{e_1}, \tau_{e_1})\rar L^1(M_{e_1}, \tau_{e_1})$. For all $ g \in G$ and $x \in M$, note that $$ {(\gamma_g|_{L^1(M_{e_1}, \tau_{e_1})})}^*( e_1xe_1)= e_1\gamma_g^*( e_1xe_1) e_1$$
		
Write $\alpha_g(y) = e_1 \gamma_g^*(y) e_1 $ for all $ y \in M_{e_1}$. For each $ g \in G$, $\alpha_g $ is a  positive contraction and for all $ y \in M_{e_1}$  we have  $$ \rho(\alpha_g(y) ) = \rho( e_1\gamma_g^*(y)e_1) =  \rho( \gamma_g^*(y)) =  \rho( y).$$  
Finally, we have $ \alpha_g^*( e_1Xe_1)= \gamma_g(e_1X e_1)$ for all $ g \in G$ and 	for all $X \in L^1(M, \tau)$. 	Thus, conclude that $( M_{e_1}, G, \alpha)$ is a non-commutative dynamical system with a invariant f.n state $\rho$. 
Now we have the following stochastic ergodic theorem.

	\begin{thm}[Stochastic Ergodic Theorem]\label{conv in e1-e2 corner} Let $M$ be a finite von Neumann algebra with a f.n trace $\tau$.  Suppose  $(L^2(M, \tau), G, \gamma)$  is  a covariant system.   Consider the projections $e_1, e_2 \in M$ as mentioned in Remark \ref{L^1-Neveu}. Then we have the following results.  
		\begin{enumerate}
			\item[(i)] For all $B \in L^1(M_{e_1}, \tau_{e_1})$, there exists $\bar{B} \in L^1(M_{e_1}, \tau_{e_1})$ such that $A_n(B)$ converges b.a.u to $\bar{B}$. Moreover,  $A_n(B)$ converges in measure to $\bar{B}$.
			\item[(ii)] For all $B \in L^1(M_{e_2}, \tau_{e_2})$, $A_n(B)$ converges to $0$ in measure.
		\end{enumerate}
	\end{thm}
	
	\begin{proof}
		\emph{(i):}
		Following the previous discussion recall that 
	$( M_{e_1}, G, \alpha)$ is a non-commutative dynamical system with a invariant f.n state $\rho$, i.e,  $\rho(\alpha_g(x)) = \rho(x)$ for all $g \in G$ and $x \in M_{e_1}$. Let $B \in L^1(M_{e_1}, \tau_{e_1})$, then it follows from Theorem \ref{bau conv theorem} that there exists a $\bar{B} \in L^1(M_{e_1}, \tau_{e_1})$ such that $A_n(B)$ converges to $\bar{B}$ in b.a.u.  Furthermore, the convergence in measure follows from Remark \ref{a.u. stronger than m}.\\
		
		\noindent \emph{ (ii):} From Corollary \ref{neveu decomposition}, it follows that there exists a weakly wandering operator $x_0 \in M_{+}$ such that $s(x_0)=e_2$. Hence $e_2 x_0 e_2= x_0$, which implies $x_0 \in M_{e_2}$.
		
		\noindent Now let $B$ be a non-zero element of $L^1(M_{e_2}, \tau_{e_2})_+$. Let us choose $0< \epsilon \leq 1$ and  $\delta > 0$. Since $e_2 = \chi_{(0,\infty)}(x_0) $, observe that there exists $m \in \N$ such that the projection $p:= \chi_{(\frac{1}{m}, \infty)}(x_0) \in M_{e_2}$ satisfies $\tau(e_2 - p) < \frac{\delta}{2} $. Now we define the projections
		\begin{align*}
		r_n:= \chi_{[\epsilon, \infty)}(p A_n(B) p), n \in \N,
		\end{align*}
		and claim that $\tau(r_n) < \delta/2$ for all $n \in \N$. Indeed, since $\frac{1}{m} p \leq x_0$ we have, $A_n(p) \leq m A_n(x_0)$ for all $n \in \N$, which implies $\norm{A_n(p)} \leq m\norm{A_n(x_0)}$. Now since $x_0$ is a weakly wandering operator, there exists $N_0 \in \N$ such that
		\begin{align*}
		\norm{A_n(p)} \leq \frac{\epsilon \delta}{2\tau(B)} \text{ for all } n \geq N_0. 
		\end{align*}
		Therefore, for all $n \in \N$ we have,
		\begin{align*}
		\tau(p A_n(B) p) = \tau(A_n(B) p)= \tau(B A_n(p)) \leq \tau(B) \norm{A_n(p)}.
		\end{align*}
		Note that $\epsilon r_n \leq p A_n(B) p$ for all $n \in \N$. Therefore, we have $\tau(r_n) \leq \frac{\delta}{2}$ for all $n \geq N_0$.  Define the projections $q_n:= p- r_n$, $n \in N$ and observe that for all $n \geq N_0$,
		\begin{align*}
		\tau(e_2 - q_n)= \tau(e_{2}-p+r_{n})=\tau(e_{2}-p)+ \tau(r_{n}) \leq \delta/2 + \delta/2 = \delta, \ n \geq N_0.
		\end{align*}
		We also note that, for all $n \in \N$
		\begin{align*}
		q_n A_n(B) q_n = q_n p A_n(B) p q_n &\leq  \chi_{[0, \epsilon]}(p A_n(B) p) (p A_n(B) p) \chi_{[0, \epsilon]}(p A_n(B) p) \\
		&\leq  \chi_{[0, \epsilon]}(p A_n(B) p).
		\end{align*}
		Hence, for all $n \in \N$ we have
		\begin{align*}
		\norm{q_n A_n(B) q_n} \leq \epsilon.
		\end{align*}
		The result for arbitrary $B \in L^1(M_{e_2}, \tau_{e_2})$ then follows from Proposition \ref{bau convergence closed under addition and mult}.
		
	\end{proof}
	
	\begin{rem}\label{rem on conv in L1}
		Let $X \in L^1(M ,\tau)$. Then as a consequence of Theorem \ref{conv in e1-e2 corner}, we get the following.
		\begin{enumerate}
			\item[(i)] There exists $\overline{X} \in L^1(M, \tau)$ such that for all $\epsilon, \delta>0$ there exists $N_0 \in \N$ and a projection $p \in M_{e_1}$ such that
			\begin{align*}
			\tau(e_1-p)< \delta/2, \text{ and, }  \norm{p(e_1A_n(X)e_1 - \overline{X})p} < \epsilon \text{ for all } n \geq N_0.
			\end{align*}
			\item[(ii)] For all $\epsilon, \delta> 0$, there exists a sequence of projections $\{ q_n \}_{n \in \N}$ in $M_{e_2}$ and $N_1 \in \N$ such that
			\begin{align*}
			\tau(e_2- q_n)< \delta/2, \text{ and, } \norm{q_n e_2 A_n(X)e_2 q_n} < \epsilon \text{ for all } n \geq N_1.
			\end{align*}
		\end{enumerate}
		Consider the following  projection 
		\begin{align*}
		r_n:= p+q_n, \ n \in \N.
		\end{align*}
		Note the for all $n \in \N$, $r_n$ is a projection in $M$ and
		\begin{align*}
		\tau(1-r_n) = \tau(e_1- p) + \tau(e_2-q_n)< \delta.
		\end{align*}
	\end{rem}
	\begin{lem}\label{rne1A_ne2rn is bdd}
		Let $X \in L^1(M ,\tau)_+$. Then there exists $N_2 \in \N$ such that for all $n \geq N_2$, $\norm{r_n e_1 A_n(X)e_2 r_n} \leq \sqrt{\epsilon(\epsilon + \norm{pe_1Ye_1p})}$ and $\norm{r_n e_2 A_n(X)e_1 r_n} \leq \sqrt{\epsilon(\epsilon + \norm{pe_1Ye_1p})}$.
	\end{lem}
	
	\begin{proof}
		Observe that for all $n \in \N$, $A_n(X) \in L^1(M,\tau)_+$ and  for all $n \geq N_0$, $pe_1 Y e_1p$ and $A_n(X) e_1 p$ are bounded operators. Then we claim that for all $n \geq N_0$, $p e_1 A_n(X)^{1/2}$ is also a bounded operator. Indeed, let $n \geq N_0$ and $\xi \in \mathcal{D}(A_n(X)e_1p)$. Then,
		\begin{align*}
			\langle  A_n(X)^{1/2} e_1 p \xi, A_n(X)^{1/2} e_1 p \xi \rangle &= \langle A_n(X)e_1 p \xi, e_1 p \xi \rangle \\
			&= \langle p e_1 A_n(X) e_1 p \xi, \xi \rangle \\
			&\leq \norm{p e_1 A_n(X) e_1 p} \norm{\xi} \\
			&= \norm{p e_1 (A_n(X)-\overline{X}) e_1 p + pe_1 \overline{X} e_1p} \norm{\xi}\\
			&\leq (\norm{p e_1 (A_n(X)-\overline{X}) e_1 p} + \norm{pe_1 \overline{X} e_1p}) \norm{\xi} \\
			&\leq (\epsilon + \norm{pe_1\overline{X}e_1p}) \norm{\xi}.
		\end{align*}
		Since, for all $n \in \N$, $\overline{\mathcal{D}(A_n(X) e_1 p)}= \CH$, we get $\norm{A_n(X)^{1/2} e_1 p} \leq \sqrt{\epsilon + \norm{pe_1Ye_1p}} $ for all $n \geq N_0$. Also we note that 
		\begin{align}\label{pe1Ahalf is bdd}
			\norm{p e_1 A_n(X)^{1/2}}= \norm{(A_n(X)^{1/2} e_1p)^*} \leq \sqrt{\epsilon + \norm{pe_1\overline{X}e_1p}} \text{ for all } n \geq N_0.
		\end{align}
		Again observe that for all $n \geq N_1$, $A_n(X) e_2 q_n$ is a bounded operator. We also claim that for all $n \geq N_1$, $A_n(X)^{1/2} e_2 q_n$ is a bounded operator. Indeed, let $n \geq N_1$ and $ \xi \in \mathcal{D}(A_n(X) e_2 q_n)$. Then,
		\begin{align*}
			\langle  A_n(X)^{1/2} e_2 q_n \xi, A_n(X)^{1/2} e_2 q_n \xi \rangle &= \langle A_n(X)e_2 q_n \xi, e_2 q_n \xi \rangle \\
			&= \langle q_n e_2 A_n(X) e_2 q_n \xi, \xi \rangle \\
			&\leq \norm{q_n e_2 A_n(X) e_2 q_n} \norm{\xi} \\
			&\leq \epsilon \norm{\xi} \text{ for all } n \geq N_1.
		\end{align*}
		Since, $\overline{\mathcal{D}(A_n(X) e_2 q_n)}= \CH$ for all $n \in \N$, we get, $\norm{A_n(X)^{1/2} e_2 q_n} \leq \sqrt{\epsilon} $ for all $n \geq N_1$.
		Now define $N_2:= \max \{N_0, N_1\}$ and note that for all $n \geq N_2$
		\begin{align*}
			\norm{r_n e_1 A_n(X)e_2 r_n} =& \norm{p e_1 A_n(X)e_2 q_n)}\\
			=& \norm{p e_1 A_n(X)^{1/2} A_n(X)^{1/2}e_2 q_n} \\
			\leq & \norm{p e_1 A_n(X)^{1/2})} \norm{A_n(X)^{1/2}e_2 q_n)} \\
			\leq & \sqrt{\epsilon(\epsilon + \norm{pe_1\overline{X}e_1p})}.
		\end{align*}
		Now since $r_n e_2 A_n(X)e_1 r_n = (r_n e_1 A_n(X)e_2 r_n)^*$ holds for all $n \in \N$, we have
		\begin{align*}
			\norm{r_n e_2 A_n(X)e_1 r_n} \leq \sqrt{\epsilon(\epsilon + \norm{pe_1\overline{X}e_1p})}.
		\end{align*}
	\end{proof}

		\begin{thm}[Strong Stochastic Ergodic Theorem]\label{conv in e1-e2 corner} Let $M$ be a finite von Neumann algebra with a f.n trace $\tau$.  Suppose  $(L^1(M, \tau), G, \gamma)$  is  a non-commutative dynamical system such that for each $g \in G$, $\gamma_g$ is lamperti operator.  	Let $X \in L^1(M, \tau)$, then there exists $\overline{X} \in L^1(M, \tau)$ such that $A_n(X)$ converges to $\overline{X} $ in measure. Further, suppose 
	$e_1, e_2 \in M$ are as in Remark \ref{L^1-Neveu},  then  $e_1   \overline{X}e_1 = \overline{X} $ and $e_2 \overline{X}e_2= 0$. 
	\end{thm}

\begin{proof}
 Let $X \in L^1(M, \tau)$, actually it is enough to prove the result for $ X \geq 0$. So, assume $X \geq 0$.  Then for all  $n \in \N$, $A_n(X) \in L^1(M, \tau)_+$  and
	\begin{align*}
		A_n(X)= e_1 A_n(X) e_1 + e_1 A_n(X) e_2 + e_2 A_n(X) e_1 + e_2 A_n(X) e_2.
	\end{align*}
	Observe that, it follows from Corollary \ref{proj-inv} that for all $n \in \N$, $e_1 A_n(X) e_1= A_n(e_1Xe_1) \in L^1(M_{e_1}, \tau_{e_1})_+$ and $e_2 A_n(X) e_2 = A_n(e_2 X e_2) \in L^1(M_{e_2}, \tau_{e_2})_+$.
	
	Let $\epsilon, \delta>0$. Consider the element $\overline{X} \in L^1(M, \tau)$ and projections  $r_n$ in $M$ as in Remark \ref{rem on conv in L1}. Let $Z:= e_1 \overline{X} e_1$ and note that for all $n \in \N$
	\begin{align*}
		r_n (A_n(X) - \overline{X}) r_n = r_n \Big( e_1 A_n(X) e_1 - e_1 \overline{X} e_1 \Big) r_n +& r_n e_1 A_n(X) e_2 r_n + r_n e_2 A_n(X) e_1 r_n \\
		& + r_n e_2 A_n(X) e_2 r_n .
	\end{align*}
	We also note that for all $n \in \N$, $r_n \Big( e_1 A_n(X) e_1 - e_1 \overline{X} e_1 \Big) r_n = p \Big( e_1 A_n(X) e_1 - e_1 \overline{X} e_1 \Big) p$ and $r_n e_2 A_n(X) e_2 r_n = q_n e_2 A_n(X) e_2 q_n$.
	
	Hence the result follows from Remark \ref{rem on conv in L1} and Lemma \ref{rne1A_ne2rn is bdd}.
\end{proof}

%
%
%

\newcommand{\etalchar}[1]{$^{#1}$}
\providecommand{\bysame}{\leavevmode\hbox to3em{\hrulefill}\thinspace}
\providecommand{\MR}{\relax\ifhmode\unskip\space\fi MR }
\providecommand{\MRhref}[2]{%
	\href{http://www.ams.org/mathscinet-getitem?mr=#1}{#2}
}
\providecommand{\href}[2]{#2}

\end{document}